\documentclass[11pt]{article}
\usepackage[letterpaper,margin=2.2cm]{geometry}
\usepackage[english]{babel}
\usepackage{amsmath, amsthm, mathtools, amssymb}
\usepackage{xcolor}
\usepackage{amsfonts}
\usepackage{graphicx}
\usepackage{epstopdf}
\usepackage{algorithm}
\usepackage{bm}
\usepackage{enumerate}
\usepackage[hidelinks]{hyperref}
\usepackage{booktabs}
\usepackage{subcaption}
\usepackage{cite}

\newcommand{\lang}{\left \langle}
\newcommand{\rang}{\right \rangle}

\newcommand{\cQ}{\mathcal{Q}}
\renewcommand{\div}{\operatorname{div}}

\newtheorem{theorem}{Theorem}
\newtheorem{lemma}{Lemma}

\def\bu {\bm{u}}
\def\bv {\bm{v}}
\def\bV {\bm{V}}
\def\bZ {\bm{Z}}
\def\bn {\bm{n}}

\def\d {\partial}
\def\eps {\varepsilon}

\def\cO {\mathcal O}
\def\cC {\mathcal C}
\def\cV {\mathcal V}
\def\R{\mathbb{R}}
\def\<{\lang}
\def\>{\rang}

\numberwithin{equation}{section}

\newcommand{\rev}[1]{#1}

\begin{document}

\title{A mortar method for the coupled Stokes-Darcy problem using the MAC scheme for Stokes and mixed finite elements for Darcy}

\author{Wietse M. Boon\footnotemark[1]
  \and
  Dennis Gl\"aser\footnotemark[2]
  \and
  Rainer Helmig\footnotemark[2]
  \and
  Kilian Weishaupt\footnotemark[2]
  \and
  Ivan Yotov\footnotemark[3]}
\renewcommand{\thefootnote}{\fnsymbol{footnote}}

\footnotetext[1]{Politecnico di Milano, Italy; \texttt{wietsemarijn.boon@polimi.it}}
\footnotetext[2]{University of Stuttgart, Germany; \texttt{dennis.glaeser@iws.uni-stuttgart.de, rainer.helmig@iws.uni-stuttgart.de, kilian.weishaupt@iws.uni-stuttgart.de}}
\footnotetext[3]{University of Pittsburgh, USA; \texttt{yotov@math.pitt.edu}}

\renewcommand{\thefootnote}{\arabic{footnote}}

\date{}

\maketitle

\begin{abstract}
A discretization method with non-matching grids is proposed for the coupled Stokes-Darcy problem that uses a mortar variable at the interface to couple the marker and cell (MAC) method in the Stokes domain with the Raviart-Thomas mixed finite element pair in the Darcy domain. Due to this choice, the method conserves linear momentum and mass locally in the Stokes domain and exhibits local mass conservation in the Darcy domain. The MAC scheme is reformulated as a mixed finite element method on a staggered grid, which allows for the proposed scheme to be analyzed as a mortar mixed finite element method. We show that the discrete system is well-posed and derive a priori error estimates that indicate first order convergence in all variables. The system can be reduced to an interface problem concerning only the mortar variables, leading to a non-overlapping domain decomposition method. Numerical examples are presented to illustrate the theoretical results and the applicability of the method.
\end{abstract}

\section{Introduction} 
\label{sec:intro}
The coupled Stokes-Darcy flow problem, which models coupled free fluid and porous media flows, has been extensively studied in recent years due to its numerous applications, including coupled surface and subsurface flows, flows through fractured or vuggy porous media, flows through industrial filters, and flows through biological tissues. The most commonly used formulation couples the two regions through continuity of normal velocity, balance of force, and the Beavers-Joseph-Saffman slip with friction interface conditions. Some of the early works on the mathematical and numerical analysis of Stokes-Darcy flows are \cite{DMQ,Disc-Quart-2004} using a pressure Darcy formulation and \cite{LSY,RivYot} using a mixed Darcy formulation. Since then, various numerical methods have been developed, see e.g., \cite{Galvis-Sarkis,Gatica-09,Gatica-11,Kar-Mar-Win,KanRiv,ArbBrun,Bernardi-etal-mortar-SD,boon2020parameter,flux-mortar}. The focus of this paper is the development and analysis of a numerical scheme that exhibits local momentum and mass conservation in the Stokes region and local mass conservation in the Darcy region, and allows for non-matching grids along the Stokes-Darcy interface. To the best of our knowledge, such method has not been previously developed in the literature. 

Our method couples the marker and cell (MAC) scheme \cite{Harlow-Welch} for Stokes with a mixed finite element (MFE) method for Darcy. The MAC scheme is a popular method in computational fluid dynamics, due to its local momentum and mass conservation properties. We restrict our attention to rectangular elements and refer to  \cite{Han-Wu,Eymard-MAC-rectangles,Kanschat,Nic,Nic-Wu,Gir-Lop,Li-Sun} for previous works on its analysis on such grids. On the other hand, the MFE method is widely used for Darcy flow, due to its local mass conservation and direct approximation of the Darcy velocity. In this paper we consider affine elements in the Darcy region, such as simplices and parallelograms. While the analysis can be carried out for any stable pair of MFE spaces of arbitrary degree, since the MAC scheme is of first order, we focus on the lowest order Raviart-Thomas spaces RT$_0$  \cite{RT}. A key feature of our method is that it allows for non-matching grids along the Stokes-Darcy interface. Such generality is important in practical applications where different spatial resolution may be needed in the two regions. We handle the non-matching grids through the use of mortar finite elements \cite{ACWY,APWY,GVY}. In particular, we introduce a mortar interface variable with the physical meaning of Darcy pressure and Stokes normal stress, which is used to impose weakly the continuity of normal velocity on the interface. The mortar variable is defined on a separate interface grid, which may differ from the traces of the subdomain grids. This further allows for the flexibility to choose the mortar finite element grid on a coarse scale, resulting in a multiscale discretization \cite{APWY,GVY}. The mortar method is suitable for the use of non-overlapping domain decomposition methods for the solution of the resulting coupled algebraic system \cite{Galvis-Sarkis-DD,Chen-Gunz-Robin,Disc-Quart-Valli-2007,Disc-Quart-2003,VasWangYot}. In particular, we present an algorithm that reduces the coupled problem to an interface problem for the mortar variable. We show that the interface problem is symmetric and positive definite and employ the conjugate gradient (CG) method for its solution. Each CG iteration requires the solution of subdomain Stokes and Darcy problems with specified normal stress for Stokes and pressure for Darcy on the interface. Therefore the solution algorithm involves only single-physics problems. This has an advantage compared to a monolithic solver for the fully coupled system, which has both larger dimension and larger condition number. 

There are several previous works that are relevant to our method. The MAC scheme for the coupled Stokes-Darcy problem has been studied in \cite{MAC-SD,Li-Rui,Rui-Sun}. The analysis in these papers is based on finite difference arguments and is restricted to matching grids on the interface. In \cite{MAC-MPFA}, a numerical method for the coupled Navier-Stokes - Darcy problem is developed, which is based on the MAC scheme in the fluid region and multipoint flux approximation (MPFA) \cite{MPFA,Edwards-Rogers} in the porous media region. The method is restricted to matching grids and numerical analysis is not presented. The method presented here can be considered as extension of the method from \cite{MAC-MPFA} to non-matching grids through the use of mortar finite elements. We further note that, while we focus on the RT$_0$ MFE method, our method and its analysis can be extended to the multipoint flux mixed finite element (MFMFE) discretization for Darcy flow \cite{WY-MPFA,AEKWY,Ing-Whe-Yot},
which is closely related to the MPFA method, using techniques developed in \cite{WheXueYot-msmortar,Song-Yotov-vegas,Song-Wang-Yotov}.

Or analysis is based on the reformulation of the MAC scheme for Stokes as a conforming MFE method \cite{Han-Wu}. In particular, a staggered grid for each component of the velocity can be formed with vertices corresponding to the degrees of freedom for this component, i.e., the midpoints of the associated edges (faces). Then a continuous bilinear (trilinear) field can be constructed for each velocity component on its staggered grid and the MAC scheme can be formulated as a conforming MFE method. This reformulation allows us to cast the MAC-MFE method as a MFE-MFE method and utilize tools from mortar MFE methods \cite{ACWY,APWY,GVY} in the analysis. 

The reminder of the paper is organized as follows. Some notation is introduced at the end of this section. The Stokes-Darcy model and its variational formulation are presented in Section~\ref{sec:the_model}. The numerical method is developed in Section~\ref{sec:pressure_mortar_method}. Its well-posedness analysis is carried out in Section~\ref{sec:well-posed}, followed by error analysis in Section~\ref{sec:error}. The non-overlapping domain decomposition algorithm is developed in Section~\ref{sec:DD}. Section~\ref{sec:numerical} is devoted to numerical experiments that illustrate the theoretical convergence results, as well as the performance and flexibility of the method applied to two challenging practical problems. Conclusions are presented in Section~\ref{sec:conclusions}.

We utilize the following notation in the paper. For a domain $\cO \subset \R^n$, $n \in \{2, 3\}$, $H^k(\cO)$, $k \ge 0$, is the standard notation for a Hilbert space equipped with a norm $\|\cdot\|_{k,\cO}$ \rev{and a seminorm $|\cdot|_{k,\cO}$}. The $L^2(\cO)$-inner product is denoted by $(\cdot,\cdot)_{\cO}$. We omit the subscript if $\cO = \Omega$. For a section of a domain boundary $G \subset \R^{n-1}$, $\<\cdot,\cdot\>_{G}$ denotes the $L^2(G)$-inner product or duality pairing.
The expression $a \lesssim b$ denotes that there exists a constant $C > 0$, independent of $a$, $b$, and the discretization parameter $h$, such that $C a \le b$. The definition of $a \gtrsim b$ is similar.

\section{The model problem and its variational formulation} 
\label{sec:the_model}

Consider an open, bounded domain $\Omega \subset \mathbb{R}^n$, $n \in \{2, 3\}$, partitioned into two disjoint subdomains $\Omega_S$ and $\Omega_D$ with interface $\Gamma = \partial{\Omega}_S \cap \partial{\Omega}_D$. Subscripts $S$ and $D$ are used, throughout this work, to denote entities related to Stokes and Darcy flow, respectively. Let $\bn_i$ denote the outward unit vector normal to
$\partial\Omega_i$, $i=S,D$. Let the symmetric gradient and the stress
be given by
\begin{align}
	\bm{\varepsilon}(\bu) &= \frac{1}{2}\left(\nabla \bu + (\nabla \bu)^T \right), &
	\bm{\sigma}_S &= 2 \mu \bm{\varepsilon}(\bu_S) - p_SI,
\end{align}
with $\mu > 0$ the viscosity. We consider the steady state Stokes-Darcy problem:
\begin{subequations}
	\begin{align}
		-\nabla \cdot \bm{\sigma}_S &= \bm{f}_S, 
		& \text{in }&\Omega_S, \label{stokes-1}\\ 
		\nabla \cdot \bu_S &= 0, 
		& \text{in }&\Omega_S, \label{stokes-2} \\
		\bu_D + \mu^{-1}K \nabla p_D &= 0, 
		& \text{in }&\Omega_D,\\ 
		\nabla \cdot \bu_D &= f_D, 
		& \text{in }&\Omega_D.
	\end{align}
The permeability $K$ is a positive-definite tensor whereas $\bm{f}_S$ and $f_D$ are given source terms. The coupling conditions on $\Gamma$ are given by mass conservation, momentum conservation, and the Beavers-Joseph-Saffman (BJS) condition, respectively:
\begin{align}
\bu_S \cdot \bn_S + \bu_D \cdot \bn_D  &= 0
& \text{on } &\Gamma,\label{eq: coupling_mass}\\
(\bm{\sigma}_S \, \bn_S)\cdot \bn_S &= -p_D
& \text{on } &\Gamma, \label{eq: coupling_mom}\\
(\bm{\sigma}_S \, \bn_S)\cdot \bm{\tau}_S
&= - \frac{\mu \alpha}{\sqrt{K_{\tau}}} \bu_S \cdot \bm{\tau}_S 
=: - \alpha_{BJS} \, \bu_S \cdot \bm{\tau}_S, 
& \text{on } &\Gamma, \label{eq:BJS}
\end{align}
where $K_\tau = (K \bm{\tau}_S) \cdot \bm{\tau}_S$ and $\alpha > 0$ is an experimentally determined coefficient. In \eqref{eq:BJS}, to simplify the notation, we have adopted notation for a one-dimensional interface $\Gamma$, with $\bm{\tau}_S$ being the unit tangential vector on $\Gamma$. In the case of a two-dimensional interface $\Gamma$, \eqref{eq:BJS} involves a sum over the the two unit tangential vectors on $\Gamma$. Finally, the following boundary conditions close the system:
\begin{align}
		\bu_S &= 0, & \text{on } & \partial \Omega_S \setminus \Gamma, \label{eq: BC essential1} \\
		\bu_D \cdot \bn_D &= 0, & \text{on } & \partial \Omega_D \setminus \Gamma.
                \label{eq: BC essential}
	\end{align}
\end{subequations}
Due the choice of boundary conditions, the source term $f_D$ must satisfy the
compatibility condition $\int_{\Omega_D} f_D = 0$. 

We proceed with the variational formulation of the Stokes-Darcy model problem. The function spaces for the velocity incorporate the essential boundary conditions \eqref{eq: BC essential1}--\eqref{eq: BC essential} and are defined
as follows:
\begin{subequations}
\begin{align}
		\bV_S &:= \left\{ \bv \in (H^1(\Omega_S))^n :\ 
		 \bv|_{\partial \Omega_S \setminus \Gamma} = 0 \right\}, \label{defn-VS}\\
		\bV_D &:= \left\{ \bv \in H(\div;\Omega_D) :\ 
		\bv\cdot\bn_D|_{\partial \Omega_D \setminus \Gamma} = 0 \right\}, \\
		\bV &:= \bV_S \times \bV_D,
\end{align}
where
\begin{align}
	H(\div;\Omega_D) &:= \{\bv \in (L^2(\Omega_D))^n: \ \nabla \cdot \bv \in L^2(\Omega_D)\}
\end{align}
equipped with the norm $\|\bv\|_{\div;\Omega_D}^2 := \|\bv\|_{\Omega_D}^2 + \|\nabla\cdot \bv\|_{\Omega_D}^2$. Second, the pressure space is naturally given by:
\begin{align}\label{W-defn}
  W &:= (W_S \times W_D)\cap L_0^2(\Omega) =
  (L^2(\Omega_S) \times L^2(\Omega_D))\cap L_0^2(\Omega) = L_0^2(\Omega),
\end{align}
where $L_0^2(\Omega)$ is the space of $L^2(\Omega)$ functions with
mean value zero. The norms in $\bV$ and $W$ are defined as
\begin{align}
\|\bv\|_V^2 &:= \|\bv_S\|_{1,\Omega_S}^2 + \|\bv_D\|_{\div;\Omega_D}^2, &
\|w\|_W &:= \|w\|.
\end{align}
Third, we introduce the Lagrange multiplier $\lambda$ to enforce \eqref{eq: coupling_mass} and \eqref{eq: coupling_mom}:
\begin{align}
	\lambda &\in \Lambda := H^{1/2}(\Gamma), &
	\lambda &= p_D = - (\bm{\sigma}_S \, \bn_S)\cdot \bn_S.
\end{align}
\end{subequations}
\rev{The space $\Lambda$ is chosen as the dual of the space $\{\bv_D \cdot \bn_D|_\Gamma : \bv_D \in \bV_D\}$. In particular, since $\bv_D \in H(\div;\Omega_D)$ and $\bv_D \cdot \bn_D = 0$ on $\partial \Omega_D \setminus \Gamma$, it holds that $\bv_D \cdot \bn_D|_\Gamma \in H^{-1/2}(\Gamma)$.}

With the function spaces defined, we continue with the variational formulation. We test the equations defined in the
free flow domain with $\bv_S \in \bV_S$ to obtain:
\begin{subequations} \label{eq: var_forms}
\begin{align}
	 -(\nabla \cdot \bm{\sigma}_S, \bv_S)_{\Omega_S} &= 
	(\bm{\sigma}_S, \nabla \bv_S)_{\Omega_S} 
	- \lang \bm{\sigma}_S \, \bn_S, \bv_S \rang_{\Gamma} 
	\nonumber\\
	&= (2 \mu \bm{\varepsilon}(\bu_S), \rev{\nabla}\bv_S)_{\Omega_S} 
	- (p_S, \nabla \cdot \bv_S)_{\Omega_S} 
	+ \lang \alpha_{BJS} \, \bu_S \cdot \bm{\tau}_S, \bv_S \cdot \bm{\tau}_S \rang_{\Gamma}
	+ \lang \lambda, \bv_S\cdot \bn_S \rang_{\Gamma} \nonumber \\
	% - \lang \bm{g}_\bm{\sigma}, \bv_S \rang_{\partial_\bm{\sigma} \Omega_S} 
	&= (\bm{f}_S, \bv_S)_{\Omega_S}. 
\end{align}

On the other hand, in the porous medium, we test Darcy's law with $\bv_D \in \bV_D$ to arrive at
\begin{align} 
	(\mu K^{-1} \bu_D, \bv_D)_{\Omega_D} 
	- (p_D, \nabla \cdot \bv_D)_{\Omega_D} 
	+ \lang \lambda, \bv_D \cdot \bn_D \rang_\Gamma 
	&= 0.
\end{align}

The Lagrange multiplier space $\Lambda$ is then used to impose flux continuity. In particular, using a test function $\xi \in \Lambda$, we impose
\begin{align}
  \lang \bu_S \cdot \bn_S + \bu_D \cdot \bn_D, \xi \rang_{\Gamma} = 0.
\end{align}
\end{subequations}
Combining equations \eqref{eq: var_forms} with the mass conservation equations, we arrive at the
variational problem: find the triplet $(\bu, p, \lambda) \in \bV \times W \times \Lambda$ such that
for all $(\bv, w, \xi) \in \bV \times W \times \Lambda$,
\begin{subequations} \label{eq: variational formulation}
\begin{align}
  (2 \mu \bm{\varepsilon}(\bu_S), \rev{\nabla}\bv_S)_{\Omega_S}
  + \lang \alpha_{BJS} \, \bu_S \cdot \bm{\tau}_S, \bv_S \cdot \bm{\tau}_S \rang_{\Gamma}
	& \nonumber\\
  - (p_S, \nabla \cdot \bv_S)_{\Omega_S}
  + \lang \lambda, \bv_S\cdot \bn_S \rang_{\Gamma} 
	&= (\bm{f}_S, \bv_S)_{\Omega_S} \\
	(\nabla \cdot \bu_S, w_S)_{\Omega_S} 
	&= 0 \\
	(\mu K^{-1} \bu_D, \bv_D)_{\Omega_D} 
  - (p_D, \nabla \cdot \bv_D)_{\Omega_D}
  + \lang \lambda, \bv_D \cdot \bn_D \rang_\Gamma 
	&= 0 \\
	(\nabla \cdot \bu_D, w_D)_{\Omega_D} 
  &= (f_D, w_D)_{\Omega_D} \\
  \lang \bu_S \cdot \bn_S + \bu_D \cdot \bn_D, \xi \rang_{\Gamma}
	&= 0.
\end{align}
\end{subequations}
Introducing the bilinear forms
%
%\begin{subequations}
  \begin{align*}
    a_S(\bu_S, \bv_S) &:= (2 \mu \bm{\varepsilon}(\bu_S), \rev{\nabla}\bv_S)_{\Omega_S}
    + \lang \alpha_{BJS} \, \bu_S \cdot \bm{\tau}_S, \bv_S \cdot \bm{\tau}_S \rang_{\Gamma}, &
    a_D(\bu_D, \bv_D) &:= (\mu K^{-1} \bu_D, \bv_D)_{\Omega_D}, \\
    a(\bu, \bv) &:= a_S(\bu_S, \bv_S) + a_D(\bu_D, \bv_D), \\
    b_i(\bv_i, w_i) &:= -(\nabla \cdot \bv_i, w_i)_{\Omega_i}, & 
    i &= S,D,\\
    b(\bv, w) &:= b_S(\bv_S, w_S) + b_D(\bv_D, w_D), \\
    b_{\Gamma}(\bv,\xi) &:= \lang \bv_S \cdot \bn_S + \bv_D \cdot \bn_D, \xi \rang_{\Gamma},
  \end{align*}
%  \end{subequations}
%
this system has the following structure:
\begin{subequations}\label{weak}
  \begin{align}
    a(\bu, \bv) + b(\bv, p) + b_{\Gamma}(\bv,\lambda) &=
    (\bm{f}_S, \bv_S)_{\Omega_S}, & \forall \bv &\in \bV, \label{weak-1} \\
    b(\bu,w) &= -(f_D, w_D)_{\Omega_D}, & \forall w &\in W, \label{weak-2} \\
    b_{\Gamma}(\bu,\xi) &= 0, & \forall \xi &\in \Lambda. \label{weak-3}
  \end{align}
\end{subequations}
\rev{The system \eqref{weak} is a symmetric two-fold saddle point problem.
Existence and uniqueness of a solution has been shown in \cite{LSY}. The solution satisfies
  $$
  	\|\bu_{S}\|_{1,\Omega_S} 
  	+ \|\bu_{D}\|_{\div;\Omega_D}
  	+ \|p\|_{\Omega} + \|\lambda\|_{H^{1/2}(\Gamma)}
  	\lesssim
  	\| \bm{f}_S \|_{-1, \Omega_S}
  	+ \| f_D \|_{\Omega_D}.
  $$
}

\section{\rev{Mortar MAC--MFE method}} 
\label{sec:pressure_mortar_method}

\begin{figure}
    \centering
    \includegraphics[width=0.35\textwidth]{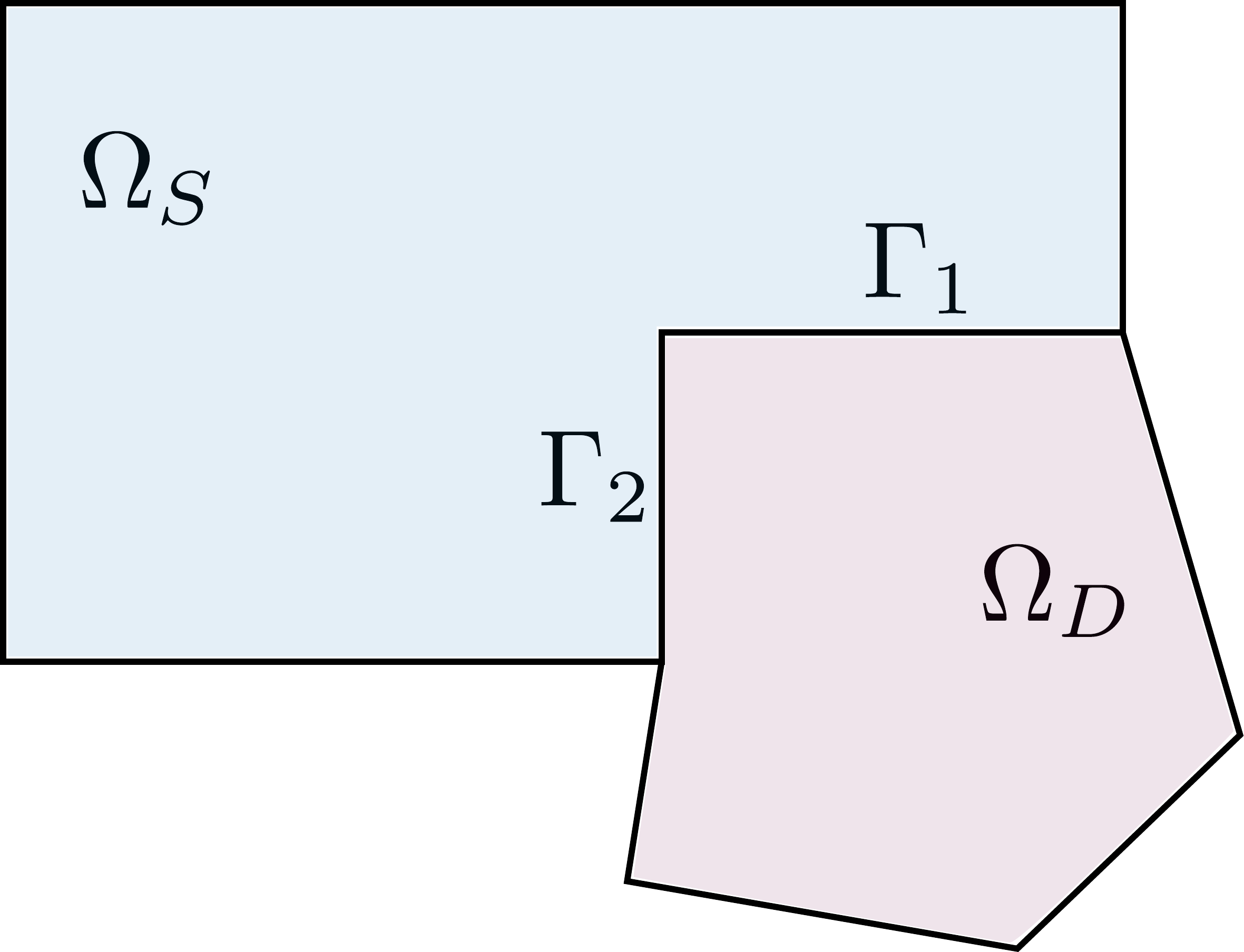}
    \caption{The Stokes-Darcy domain.}
    \label{fig:domain}
\end{figure}

Assume that the subdomains $\Omega_S$ and $\Omega_D$ are polytopal and 
consider shape-regular meshes on $\Omega_S$ and $\Omega_D$ denoted by
$\Omega_{S,h}$ and $\Omega_{D, h}$, respectively. The two meshes may
be non-matching on the interface $\Gamma$. In the Stokes region we
consider the MAC scheme \cite{Harlow-Welch}, described below, and restrict
$\Omega_{S,h}$ to rectangular-type elements. \rev{Due to this restriction, $\Gamma$ is a piecewise linear curve consisting of horizontal and vertical segments, denoted by $\Gamma^1$ and $\Gamma^2$, respectively, see Figure~\ref{fig:domain}.} The Darcy mesh $\Omega_{D, h}$ may consist of affine elements. Let $\bV_{D,h}\times W_{D,h} \subset \bV_D\times W_D$ be mixed finite element
spaces that form a stable pair for the Darcy sub-problem. Even though
theoretically arbitrary order spaces may be used, since the MAC scheme
is of first order, we focus on the lowest order Raviart-Thomas
spaces RT$_0$ \cite{RT} for $\bV_{D, h}$ and the piecewise constants for $W_{D, h}$.
We emphasize that this pair of spaces has the property:
\begin{align} \label{eq: div property}
	\nabla \cdot \bV_{D, h} = W_{D, h}.
\end{align}

\rev{The Lagrange multiplier space $\Lambda$ in \eqref{weak} is discretized as follows.}
We consider a tessellation of $\Gamma$ denoted by $\Gamma_h$, which can be constructed independently of the previously introduced meshes. Let
$\Lambda_h$ be the discretization of $\Lambda$ consisting of
(dis)continuous, piecewise polynomials. For simplicity of the presentation we consider the mortar grid on the same scale $h$ as the traces of the subdomain grids. The analysis can be extended to a multiscale setting with the mortar grid defined on a coarse scale $H$, utilizing multiscale mortar finite element techniques developed in \cite{APWY,GVY}. 

We next describe the MAC scheme used in the Stokes region. The pressure $p_S$ is computed at the centers of the elements of $\Omega_{S,h}$. The normal velocities $\bu_S\cdot\bn$ are computed at the centers of the edges (faces) of the elements. For example, in two dimensions these are the horizontal velocities $u_{S,1}$ at the midpoints of the vertical edges, and the vertical velocities $u_{S,2}$ at the midpoints of the horizontal edges, see Figure~\ref{fig-MAC} (left). We note that these are the same as the degrees of freedom of the RT$_0$ spaces. \rev{We denote the discrete MAC velocity and pressure spaces as $\bV_S^{\rm MAC}$ and $W_S^{\rm MAC}$, respectively.} For each edge we consider an associated control volume obtained by drawing lines parallel to the edge through the centers of the two neighboring elements. If an edge is on the boundary, it is associated with a half-volume. We denote a generic control volume by $G_i$, with $i=1$ for vertical edges and $i=2$ for horizontal edges, see Figure~\ref{fig-MAC} (left). The momentum balance \eqref{stokes-1} is imposed component-wise: $-\nabla\cdot \bm{\sigma}_{S,i} = \bm{f}_{S,i}$, $i=1,2$, where $\bm{\sigma}_{S,i}$ is the $i$-th row of $\bm{\sigma}_{S}$. The divergence theorem gives
\begin{align*}
-\int_{\partial G_i} \bm{\sigma}_{S,i} \cdot\bn = \int_{G_i} \bm{f}_{S,i},
\end{align*}
where $\bn$ is the unit outward normal vector to $G_i$.
Taking $i=1$, and using the notation from Figure~\ref{fig-MAC}, we obtain
\begin{align}\label{MAC-momentum-1}
-\int_{\partial G_1} \bm{\sigma}_{S,1} \cdot\bn
& = - \int_{\partial G_1} \left(2\mu\begin{pmatrix} \varepsilon_{11}\\\varepsilon_{12}
\end{pmatrix} - \begin{pmatrix} p_S \\ 0 \end{pmatrix} \right)\cdot \bn \nonumber \\
& =
- \int_{l_1}(-2\mu \varepsilon_{11} + p_S)
- \int_{r_1}(2\mu \varepsilon_{11} - p_S)
-\int_{b_1} -2\mu \varepsilon_{12}
-\int_{t_1} 2\mu \varepsilon_{12}
.
\end{align}
Similarly,
\begin{align}\label{MAC-momentum-2}
  -\int_{\partial G_2} \bm{\sigma}_{S,2} \cdot\bn
& = - \int_{\partial G_1} \left(2\mu\begin{pmatrix} \varepsilon_{21}\\\varepsilon_{22}
\end{pmatrix} - \begin{pmatrix} 0 \\ p_S \end{pmatrix} \right)\cdot \bn \nonumber \\  
  & =
- \int_{l_2} -2\mu \varepsilon_{21} 
- \int_{r_2} 2\mu \varepsilon_{21}
  -\int_{b_2} (-2\mu \varepsilon_{22} + p_S)
  -\int_{t_2} (2\mu \varepsilon_{22} - p_S)
.
\end{align}

\begin{figure}
    \centering
    \includegraphics[width=0.99\textwidth]{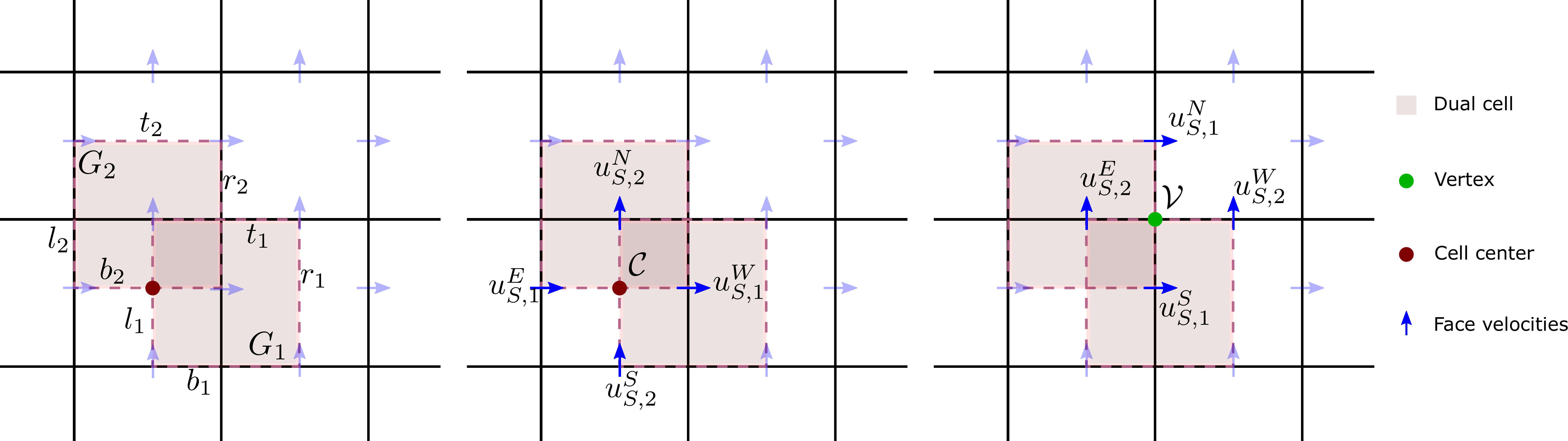}
    \caption{Grids and velocity degrees of freedom for the MAC scheme. }
    \label{fig-MAC}
\end{figure}

For full volumes $G_i$, the edge integrals in \eqref{MAC-momentum-1} and \eqref{MAC-momentum-2} are approximated by the midpoint rule, therefore $\varepsilon_{11}$, $\varepsilon_{22}$, and $p_S$ are evaluated at the centers of the primal cells, while $\varepsilon_{12}$ and $\varepsilon_{21}$ are evaluated at the vertices of the primal cells. Since
$\eps_{11} = \frac{\d u_{S,1}}{\d x}$, $\eps_{22} = \frac{\d u_{S,2}}{\d y}$,
and $\eps_{12} = \eps_{21} = \frac12\left(\frac{\d u_{S,1}}{\d y} + \frac{\d u_{S,2}}{\d x} \right)$, this implies that $\frac{\d u_{S,1}}{\d x}$,
$\frac{\d u_{S,2}}{\d y}$, and $p_S$ are evaluated at the cell centers, while
$\frac{\d u_{S,1}}{\d y}$ and $\frac{\d u_{S,2}}{\d x}$ are evaluated at the vertices. The cell-centered values are degrees of freedom for the pressure. For the velocity derivatives, using the notation from Figure~\ref{fig-MAC} (center, right), the quantities are approximated as
\begin{align}
\frac{\d u_{S,1}}{\d x}(\mathcal C) &= \frac1h(u_{S,1}^E - u_{S,1}^W), &
\frac{\d u_{S,2}}{\d y}(\mathcal C)  = \frac1h(u_{S,2}^N - u_{S,2}^S), \label{MAC-vel-1}\\
\frac{\d u_{S,1}}{\d y}(\mathcal V) &= \frac1h(u_{S,1}^N - u_{S,1}^S), &
\frac{\d u_{S,2}}{\d x}(\mathcal V)  = \frac1h(u_{S,2}^E - u_{S,2}^W),
\label{MAC-vel-2}
\end{align}
where for simplicity we have assumed that the mesh is uniform.

\begin{figure}
  \centering
  \includegraphics[width=0.7\textwidth]{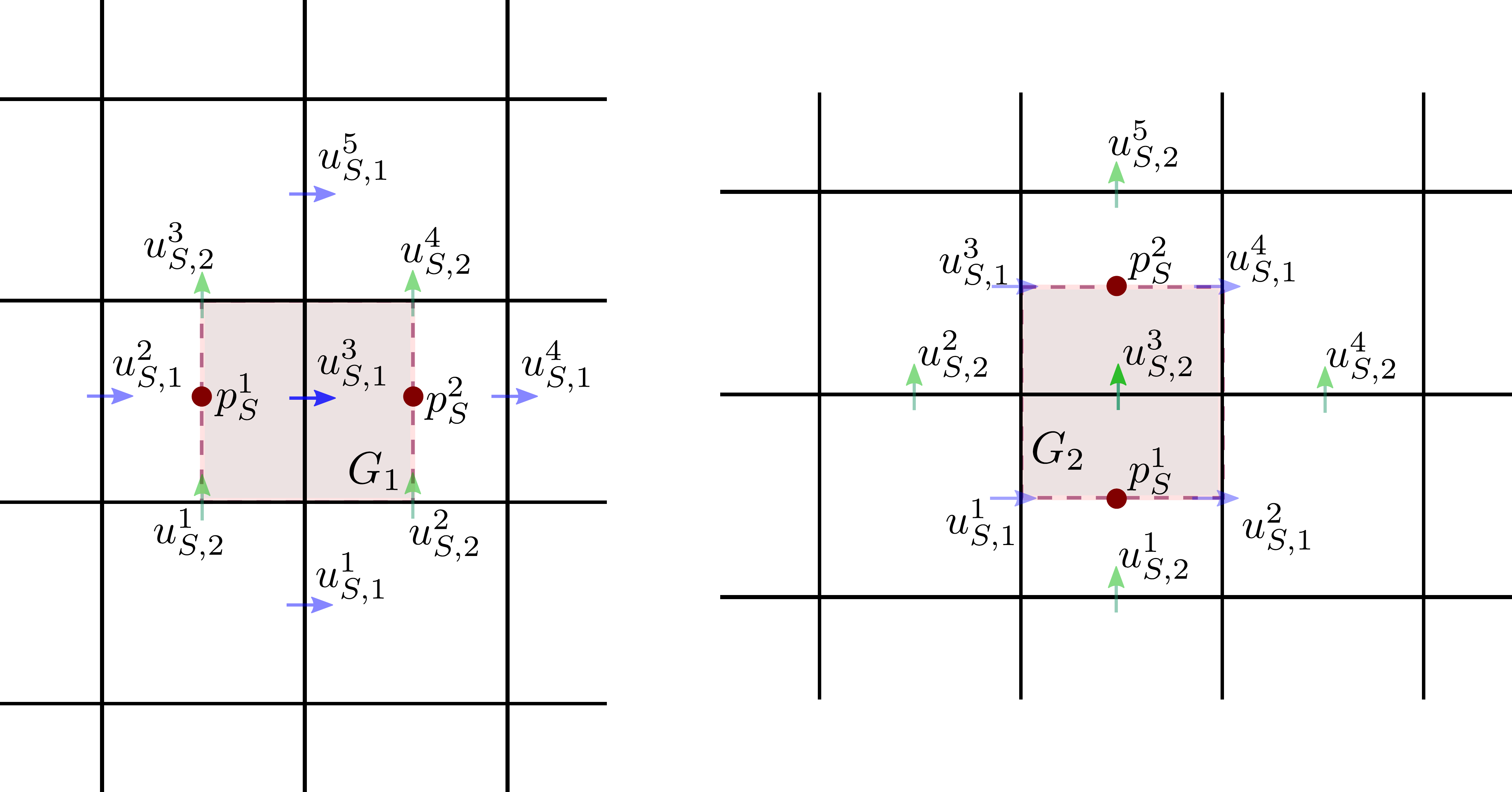}
\caption{Finite difference stencils for the MAC momentum balance equations on volumes $G_1$ (left) and $G_2$ (right).}
\label{fig-MAC-momentum}
\end{figure}  

\rev{
Referring to the notation from Figure~\ref{fig-MAC-momentum} and using \eqref{MAC-vel-1}--\eqref{MAC-vel-2}, the momentum balance equations \eqref{MAC-momentum-1}--\eqref{MAC-momentum-2} on interior volumes $G_1$ and $G_2$ give, respectively,
  \begin{align}
    & 2\mu\left(-u_{S,1}^2 + 2u_{S,1}^3 - u_{S,1}^4 + \frac12(-u_{S,1}^1 + 2u_{S,1}^3 - u_{S,1}^5)
    + \frac12(u_{S,2}^2 - u_{S,2}^1) - \frac12(u_{S,2}^4 - u_{S,2}^3) \right) \nonumber \\
    & \qquad + h\left(p_S^2 - p_S^1\right) = \int_{G_1} \bm{f}_{S,1}, \label{G1} \\
    & 2\mu\left(-u_{S,2}^1 + 2u_{S,2}^3 - u_{S,2}^5 + \frac12(-u_{S,2}^2 + 2u_{S,2}^3 - u_{S,2}^4)
    + \frac12(u_{S,1}^3 - u_{S,1}^1) - \frac12(u_{S,1}^4 - u_{S,1}^2) \right) \nonumber \\
    & \qquad + h\left(p_S^2 - p_S^1\right) = \int_{G_2} \bm{f}_{S,2}. \label{G2} 
  \end{align}
}

The mass balance \eqref{stokes-1} is imposed on the primal cells $E$:
\begin{equation}\label{MAC-mass}
\int_{\d E} \bu_S\cdot\bn = h(u_{S,1}^E - u_{S,1}^W + u_{S,2}^N - u_{S,2}^S) = 0.
\end{equation}

\rev{We next discuss briefly the MAC discretization of the boundary conditions. The condition $\bu_S\cdot\bn_S$ is essential, since the MAC degrees of freedom include the normal velocities on the boundary. In this case, the momentum balance equation \eqref{MAC-momentum-1} or \eqref{MAC-momentum-2} on the associated half-volume is omitted. The condition $\bu_S\cdot\bm{\tau}_S$ is natural, as the term appears in the momentum balance equations for volumes $G_1$ adjacent to horizontal boundaries, through the integrals $\int_{b_1} -2\mu\varepsilon_{12}$ and
  $\int_{t_1} 2\mu \varepsilon_{12}$ in \eqref{MAC-momentum-1}, and volumes $G_2$ adjacent to vertical boundaries, through the integrals $\int_{l_2} -2\mu \varepsilon_{21}$ and $\int_{r_2} 2\mu \varepsilon_{21}$ in \eqref{MAC-momentum-2}. In particular, as the vertex $\mathcal V$ in Figure~\ref{fig-MAC} (right) is on the boundary, one or both of the expressions in \eqref{MAC-vel-2} are modified to involve the boundary value. For example, on a bottom boundary, the first equation in \eqref{MAC-vel-2} becomes $\frac{\d u_{S,1}}{\d y}(\mathcal V) = \frac2h(u_{S,1}^N - u_{S,1}(\mathcal V))$, which results in $u_{S,1}^1$ not being included in \eqref{G1}. Finally, both stress boundary conditions $(\bm{\sigma}_S \, \bn_S)\cdot \bn_S$ and $(\bm{\sigma}_S \, \bn_S)\cdot \bm{\tau}_S$ are natural. In particular, $(\bm{\sigma}_S \, \bn_S)\cdot \bn_S$ appears in the integrals $\int_{l_1}(-2\mu \varepsilon_{11} + p_S)$ and $\int_{r_1}(2\mu \varepsilon_{11} - p_S)$ in \eqref{MAC-momentum-1} on half-volumes $G_1$ adjacent to vertical boundaries, as well as in the integrals $\int_{b_2} (-2\mu \varepsilon_{22} + p_S)$ and 
  $\int_{t_2} (2\mu \varepsilon_{22} - p_S)$ in \eqref{MAC-momentum-2} on half-volumes $G_2$ adjacent to horizontal boundaries. For example, on a left boundary, $u_{S,1}^2$ and $p_S^1$ are not included in \eqref{G1}. Similarly, $(\bm{\sigma}_S \, \bn_S)\cdot \bm{\tau}_S$ appears in the integrals $\int_{b_1} -2\mu\varepsilon_{12}$ and $\int_{t_1} 2\mu \varepsilon_{12}$ in \eqref{MAC-momentum-1} on volumes $G_1$ adjacent to horizontal boundaries, as well as in integrals $\int_{l_2} -2\mu \varepsilon_{21}$ and $\int_{r_2} 2\mu \varepsilon_{21}$ in \eqref{MAC-momentum-2} on volumes $G_2$ adjacent to vertical boundaries. For example, on a bottom boundary, $u_{S,1}^1$ is not included in \eqref{G1}. 
}

\begin{figure}
    \centering
    \includegraphics[width=0.8\textwidth]{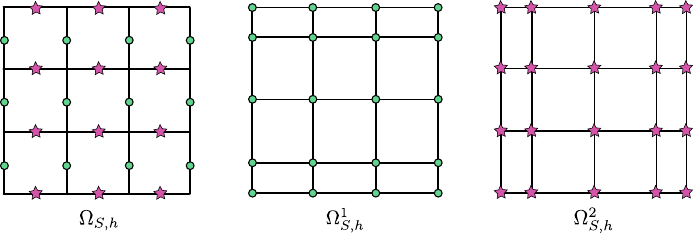}
    \caption{Primal MAC grid with velocity degrees of freedom (left) and staggered grids with degrees of freedom for the horizontal (center) and vertical (right) velocities.}
    \label{fig-MAC-MFE}
\end{figure}

\rev{
  We are now ready to formulate the mortar MAC-MFE method for the approximation of the variational problem \eqref{eq: variational formulation}: find $(\bu_S,p_S) \in \bV_S^{\rm MAC} \times W_S^{\rm MAC}$,
$(\bu_{D,h},p_{D,h}) \in \bV_{D,h}\times W_{D, h}$, and $\lambda_h \in \Lambda_h$ such that
\begin{subequations} \label{MAC-MFE}
\begin{align}
  & (\bu_S,p_S) \mbox{ satisfy the MAC equations \eqref{MAC-momentum-1}--\eqref{MAC-mass} in } \Omega_S \mbox{ with} \nonumber \\
  & \bu_S = 0 \text{ on } \partial \Omega_S \setminus \Gamma,
  (\bm{\sigma}_S \, \bn_S)\cdot \bn_S = - \lambda_h \text{ on } \Gamma,
  (\bm{\sigma}_S \, \bn_S)\cdot \bm{\tau}_S = - \alpha_{BJS} \, \bu_S \cdot \bm{\tau}_S \text{ on } \Gamma,  \label{MAC-scheme}\\
  & (\mu K^{-1} \bu_{D,h}, \bv_{D,h})_{\Omega_D} 
  - (p_{D,h}, \nabla \cdot \bv_{D,h})_{\Omega_D}
  + \lang \lambda_h, \bv_{D,h} \cdot \bn_D \rang_\Gamma = 0 \quad \forall \bv_{D,h} \in \bV_{D,h},\\
& (\nabla \cdot \bu_{D,h}, w_{D,h})_{\Omega_D} = (f_D, w_{D,h})_{\Omega_D} \quad \forall w_{D,h} \in W_{D,h},\\
  & \lang \bu_S \cdot \bn_S + \bu_{D,h} \cdot \bn_D, \xi_h \rang_{\Gamma} = 0 \quad \forall \xi_h \in \Lambda_h,\label{MAC-MFE-Gamma}
\end{align}
\end{subequations}
where $\bu_S\cdot\bn_S$ in \eqref{MAC-MFE-Gamma} is interpreted as a piecewise constant function associated with the MAC degrees of freedom on $\Gamma$. 
}

For the purpose of the analysis, we will utilize the reformulation of the MAC scheme as a conforming mixed finite element method for Stokes \cite{Han-Wu}. For
simplicity of the presentation, we focus on the two dimensional
case. The extension to three dimensions is natural. Starting from the primal grid and degrees of freedom, Figure~\ref{fig-MAC-MFE} (left), we consider two staggered grids $\Omega_{S, h}^i$, $i=1,2$, for the horizontal and vertical velocities, respectively, with vertices associated with their respective degrees of freedom, see Figure~\ref{fig-MAC-MFE} (center, right). Note that degrees of freedom have been included for the tangential velocity on the boundary of $\Omega_S$. The boundary velocities are determined from the Dirichlet boundary condition on the external boundary and are incorporated into the stress interface conditions on $\Gamma$. The values at the vertices allow
for constructing continuous bilinear functions on the two staggered
grids. Denote the corresponding spaces by $S_h^i$, $i=1,2$. Let
$\bV_{S,h} = (S_h^1\times S_h^2)\cap\bV_S$. \rev{We emphasize that, due to \eqref{defn-VS}, $\bv_{S,h} \in \bV_{S,h}$ satisfies $\bv_{S,h} = 0$ on $\partial \Omega_S \setminus \Gamma$.} The Stokes mixed
finite element pair is $\bV_{S,h}\times W_{S,h}$, where $W_{S,h}$
is the space of piecewise constant functions on the primal grid
$\Omega_{S, h}$.

\rev{For $\bu_{S,h} \in \bV_{S,h}$, let $\widetilde{\bm{\varepsilon}}(\bu_{S,h})$ be a modification of $\bm{\varepsilon}(\bu_{S,h})$ with
$$  
  \widetilde{\bm{\varepsilon}}(\bu_{S,h})_{12} = \frac{1}{2}\left(\frac{\d u_{S,h,1}}{\d y} + Q_1\frac{\d u_{S,h,2}}{\d x} \right) \quad \mbox{and} \quad
 \widetilde{\bm{\varepsilon}}(\bu_{S,h})_{21} = \frac{1}{2}\left(Q_2\frac{\d u_{S,h,1}}{\d y} + \frac{\d u_{S,h,2}}{\d x} \right), 
$$
 where $Q_1\frac{\d u_{S,h,2}}{\d x}$ and $Q_2\frac{\d u_{S,h,1}}{\d y}$ are defined as follows. For $Q_1$, consider $E_2 \in \Omega_{S,h}^2$ and split it as $E_2^b \cup E_2^t$ by the horizontal edges from $\Omega_{S,h}^1$. Noting that $\frac{\d u_{S,h,2}}{\d x} = \alpha + \beta y$ on $E_2$, we define $Q_1\frac{\d u_{S,h,2}}{\d x}|_{E_2}$ as the piecewise constant function satisfying $Q_1\frac{\d u_{S,h,2}}{\d x}|_{E_2^b} = \frac{\d u_{S,h,2}}{\d x}|_b$ and
$Q_1\frac{\d u_{S,h,2}}{\d x}|_{E_2^t} = \frac{\d u_{S,h,2}}{\d x}|_t$, where $b$ and $t$ denote the bottom and top edges of $E_2$, respectively. We define $Q_2\frac{\d u_{S,h,1}}{\d y}$ on $E_1 \in \Omega_{S,h}^1$ similarly, by splitting it by the vertical edges from $\Omega_{S,h}^2$.
}

\rev{
  For each element $E_i \in \Omega_{S, h}^i$, $i=1,2$, let $Q^1(E_i)$ denote the space of bilinear functions on $E_i$. For a function $\varphi$ with domain $E_i$ such that $\varphi$ is well defined at the vertices, let $I_{E_i}(\varphi) \in Q^1(E_i)$ interpolate $\varphi$ at the four vertices of $E_i$.} For each element $E \in \Omega_{S,h}$, let $I_{E}:C^0(E) \to P_0(E)$  \rev{interpolate} the function at the center of $E$, where $P_0(E)$ denotes the space of constant functions on $E$.
Motivated by \cite{Han-Wu}, we define the discrete bilinear forms
\begin{subequations}
\begin{align}
  & a_{S,h}(\bu_{S,h},\bv_{S,h}) :=
  \sum_{E_1\in \Omega_{S, h}^1}\int_{E_1} 2\mu I_{E_1}(\widetilde{\bm{\varepsilon}}(\bu_{S,h})_1
  \cdot (\rev{\nabla}\bv_{S,h})_1)
  + \sum_{E_2\in \Omega_{S, h}^2}\int_{E_2} 2\mu I_{E_2}(\widetilde{\bm{\varepsilon}}(\bu_{S,h})_2
  \cdot (\rev{\nabla}\bv_{S,h})_2) \nonumber \\
  & \quad + \sum_{E_1\in \Omega_{S, h}^1}\int_{\d E_1 \cap\Gamma^1} \alpha_{BJS}
  I_{E_1}(u_{S,h,1} \, v_{S,h,1})
  + \sum_{E_2\in \Omega_{S, h}^2}\int_{\d E_2 \cap\Gamma^2} \alpha_{BJS}
  I_{E_2}(u_{S,h,2} \, v_{S,h,2}), \label{a-sh}\\
  & b_{S,h}(\bv_{S,h},w_{S,h}) := - \sum_{E\in \Omega_{S, h}}\int_E
  I_E(\nabla\cdot \bv_{S,h}) w_{S,h}, \label{b-sh}
\end{align}
\end{subequations}
where $\bm{\varepsilon}(\bu_{S,h})_i$ is the $i$-th row of $\bm{\varepsilon}(\bu_{S,h})$, $(\nabla\bv_{S,h})_i$ is the $i$-th row of $\nabla\bv_{S,h}$, \rev{and we recall that $\Gamma^1$ and $\Gamma^2$ are, respectively, the horizontal and vertical parts of the interface $\Gamma$, see Figure~\ref{fig:domain}.
}

\rev{
	Finally, for incorporating the right-hand side, we define the interpolants $\cQ_i$, $i=1,2$, for $\varphi \in S_h^i$ such that $\cQ_i \varphi$ is constant on each control volume $G_i$, defined from the value of $\varphi$ at the vertex of $\Omega_{s, h}^i$ interior to $G_i$. The combined interpolant is denoted by $\cQ = (\cQ_1, \cQ_2)$.
}

\rev{
\begin{lemma}\label{lem:MAC=MFE}
  The MAC scheme \eqref{MAC-scheme} is equivalent to the following mixed finite element method: find $(\bu_{S,h},p_{S,h}) \in \bV_{S,h}\times W_{S,h}$ such that
  \begin{subequations}
  \begin{align}
    a_{S,h}(\bu_{S,h},\bv_{S,h}) + b_{S,h}(\bv_{S,h},p_{S,h}) + \lang \lambda_h,\bv_{S,h}\cdot\bn_S \rang_{\Gamma} 
    &=  (\bm{f}_S, \cQ \bv_{S,h})_{\Omega_S}, & \forall \bv_{S,h} \in \bV_{S,h}, \label{MAC-1}\\
    b_{S,h}(\bu_{S,h},w_{S,h}) &= 0, & \forall w_{S,h} \in W_{S,h}. \label{MAC-2}
  \end{align}
  \end{subequations}
\end{lemma}
}
\rev{
  \begin{proof}
A simple calculation shows that \eqref{MAC-1} with $\bv_{S,h}$ the basis function in $S_h^1$ associated with the vertex of $\Omega_{S,h}^1$ at the degree of freedom $u_{S,1}^3$ in Figure~\ref{fig-MAC-momentum} (left) results in \eqref{G1}. Similarly, \eqref{MAC-1} with $\bv_{S,h}$ the basis function in $S_h^2$ associated with the vertex of $\Omega_{S,h}^2$ at the degree of freedom $u_{S,2}^3$ in Figure~\ref{fig-MAC-momentum} (right) results in \eqref{G2}. 

One can also check that, adjacent to the external boundary and the interface $\Gamma$, the MAC equations \eqref{MAC-momentum-1}--\eqref{MAC-momentum-2} and the MFE equation \eqref{MAC-1} result in the same modification of \eqref{G1} and \eqref{G2}. In particular,
since the stress interface conditions in \eqref{MAC-scheme} are natural, the summation of \eqref{MAC-momentum-1} and \eqref{MAC-momentum-2} results in the interface terms
\begin{equation}\label{MAC-interface}
  \int_\Gamma -(\bm{\sigma}_S \, \bn_S)\cdot \bn_S = \int_\Gamma \lambda_h \quad \mbox{and} \quad
  \int_\Gamma -(\bm{\sigma}_S \, \bn_S)\cdot \bm{\tau}_S = \int_\Gamma \alpha_{BJS} \, \bu_S \cdot \bm{\tau}_S,
\end{equation}
which correspond to the interface terms that appear in \eqref{MAC-1}. 
Finally, \eqref{MAC-2} with $w_{S,h}$ the basis function in $W_{S,h}$ associated with element $E$ results in the mass balance equation \eqref{MAC-mass}. 
\end{proof}
}

\rev{The equivalence established in Lemma~\ref{lem:MAC=MFE} allows us to rewrite the mortar MAC-MFE method \eqref{MAC-MFE} as a mortar mixed finite element method.} Let 
$\bV_h := \bV_{S, h} \times \bV_{D, h}$, $W_h := W_{S, h} \times W_{D, h}$, 
\begin{align*}
	a_h(\bu_h; \bv_h) &:= a_{S,h}(\bu_{S,h},\bv_{S,h})
	+ a_D(\bu_{D,h};\bv_{D,h}), &
	b_h(\bv_h,w_h) &:= b_{S,h}(\bv_{S,h},w_{S,h}) + b_{D}(\bv_{D,h},w_{D,h}).
\end{align*}
\rev{Due to Lemma~\ref{lem:MAC=MFE}, the mortar MAC-MFE method \eqref{MAC-MFE} is equivalent to the following mortar MFE method:} find
$(\bu_h, p_h, \lambda_h) \in \bV_h \times W_h \times \Lambda_h$ such that
\begin{subequations}\label{weak-h}
  \begin{align}
    a_h(\bu_h, \bv_h) + b_h(\bv_h, p_h) + b_{\Gamma}(\bv_h,\lambda_h) &=
    (\bm{f}_S, \rev{\cQ} \bv_{S,h})_{\Omega_S}, &
    \forall \bv_h &\in \bV_h, \label{weak-h-1} \\
    b_h(\bu_h,w_h) &= -(f_D, w_{D,h})_{\Omega_D}, &
    \forall w_h &\in W_h, \label{weak-h-2} \\
    b_{\Gamma}(\bu_h,\xi_h) &= 0, &
    \forall \xi_h &\in \Lambda_h. \label{weak-h-3}
  \end{align}
\end{subequations}

\section{Well posedness}
\label{sec:well-posed}

We begin with stating results from the literature for interpolants in
the Stokes and Darcy velocity spaces and local inf-sup stability that
will be used in the analysis. It is shown in \cite{Han-Wu} that there
exists an interpolant $\Pi_{S,h}:\bV_S \to \bV_{S,h}$, where
$\Pi_{S,h}\bv_S = (\Pi_{S,h}^1 v_{S,1},\Pi_{S,h}^2 v_{S,2}) \in
S_h^1\times S_h^2$ such that for all sufficiently smooth $\bv_S \in
\bV_S$,
\begin{subequations}\label{interp-MAC}
  \begin{align}
    b_{S,h}(\Pi_{S,h}\bv_S, w_{S,h}) &= b_S(\bv_S,w_{S,h}),
    \quad \forall \,w_{S,h} \in W_{S,h}, \label{PiSh-div} \\
    \|\bv_S - \Pi_{S,h}\bv_S\|_{1,\Omega_S} &\lesssim h |\bv_S|_{2,\Omega_S},
    \label{PiSh-approx}\\
    \|\Pi_{S,h}\bv_S\|_{1,\Omega_S} &\lesssim \|\bv_S\|_{1,\Omega_S}.
    \label{PiSh-cont}
    \end{align}
\end{subequations}
Furthermore, the following continuity and inf-sup condition hold:
\begin{subequations}
  \begin{align}
    b_{S,h}(\bv_{S,h},w_{S,h}) &\lesssim \|\bv_{S,h}\|_{1,\Omega_S}\|w_{S,h}\|_{\Omega_S}, \quad
    && \forall \, \bv_{S,h} \in \bV_{S,h}, \, w_{S,h} \in W_{S,h}, \label{cont-b-MAC} \\
    \sup_{\bv_{S,h} \in \bV_{S,h}\setminus 0}\frac{b_{S,h}(\bv_{S,h},w_{S,h})}{\|\bv_{S,h}\|_{1,\Omega_S}}
    &\gtrsim \|w_{S,h}\|_{\Omega_S}, \quad
    && \forall w_{S,h} \in W_{S,h}.
    \label{inf-sup-MAC}
\end{align}
\end{subequations}

\rev{We next establish continuity and coercivity for the bilinear form $a_{S,h}(\bu_{S,h},\bv_{S,h})$. Let
$$
  \bZ_{S,h} = \{\bv_{S,h} \in \bV_{S,h}: b_{S,h}(\bv_{S,h},w_{S,h}) = 0 \ \ \forall w_{S,h} \in W_{S,h}\}.
$$  
\begin{lemma}\label{lem:a-MAC}
It holds that
\begin{subequations} 
 \begin{align}
    a_{S,h}(\bu_{S,h},\bv_{S,h}) &\lesssim \|\bu_{S,h}\|_{1,\Omega_S}\|\bv_{S,h}\|_{1,\Omega_S}, \quad
    && \forall \, \bu_{S,h},\bv_{S,h} \in \bV_{S,h},
    \label{cont-a-MAC}    \\
    a_{S,h}(\bv_{S,h},\bv_{S,h}) &\gtrsim \|\bv_{S,h}\|_{1,\Omega_S}^2, \quad
    && \forall \, \bv_{S,h} \in \bZ_{S,h}. \label{coer-a-MAC}
\end{align}
\end{subequations}
\end{lemma}
\begin{proof}
The continuity bound \eqref{cont-a-MAC} follows easily from the definition \eqref{a-sh}. For the coercivity bound, consider the equation \eqref{G1} with $\bu_{S,h} \in \bZ_{S,h}$. Since $u_{S,2}^2 - u_{S,2}^4 =  u_{S,1}^4 - u_{S,1}^3$ and $u_{S,2}^3 - u_{S,2}^1 =  u_{S,1}^2 - u_{S,1}^3$, we obtain that for the choice of $\bv_{S,h}$ in \eqref{G1}, 
$$
a_{S,h}(\bu_{S,h},\bv_{S,h}) = \mu(-u_{S,1}^2 + 2u_{S,1}^3 - u_{S,1}^4) + \mu(-u_{S,1}^1 + 2u_{S,1}^3 - u_{S,1}^5) = \sum_{E_1\in \Omega_{S, h}^1}\int_{E_1} \mu I_{E_1}((\nabla \bu_{S,h})_1
  \cdot (\nabla\bv_{S,h})_1).
$$
Similarly, for the choice of $\bv_{S,h}$ in \eqref{G2}, 
$$
a_{S,h}(\bu_{S,h},\bv_{S,h}) = \mu(-u_{S,2}^1 + 2u_{S,2}^3 - u_{S,2}^5) + \mu(-u_{S,2}^2 + 2u_{S,2}^3 - u_{S,2}^4) = \sum_{E_2\in \Omega_{S, h}^2}\int_{E_2} \mu I_{E_2}((\nabla \bu_{S,h})_2
  \cdot (\nabla\bv_{S,h})_2).
$$
  A similar modification holds for a test function $\bv_{S,h}$ with support adjacent to $\partial \Omega_1$, implying that for $\bu_{S,h} \in \bZ_{S,h}$
\begin{align}
& a_{S,h}(\bu_{S,h},\bv_{S,h}) =
  \sum_{E_1\in \Omega_{S, h}^1}\int_{E_1} \mu I_{E_1}((\nabla\bu_{S,h})_1
  \cdot (\nabla\bv_{S,h})_1)
  + \sum_{E_2\in \Omega_{S, h}^2}\int_{E_2} \mu I_{E_2}((\nabla\bu_{S,h})_2
  \cdot (\nabla\bv_{S,h})_2) \nonumber \\
  & \quad + \sum_{E_1\in \Omega_{S, h}^1}\int_{\d E_1 \cap\Gamma^1} \alpha_{BJS}
  I_{E_1}(u_{S,h,1} \, v_{S,h,1})
  + \sum_{E_2\in \Omega_{S, h}^2}\int_{\d E_2 \cap\Gamma^2} \alpha_{BJS}
  I_{E_2}(u_{S,h,2} \, v_{S,h,2}). \label{a-sh-new}
\end{align}
Therefore, noting that $\displaystyle\int_{E_i} I_{E_i}(\cdot)$ corresponds to employing the vertex quadrature rule, a simple calculation, see \cite[Lemma~2.4]{WY-MPFA}, gives that for all $\bv_{S,h} \in \bZ_{S,h}$
$$
a_{S,h}(\bv_{S,h},\bv_{S,h}) \gtrsim \|\nabla\bv_{S,h}\|_{\Omega_S}^2 \gtrsim \|\bv_{S,h}\|_{1,\Omega_S}^2,
$$
where the last inequality follows from the Poincar\'e inequality.
\end{proof}
}

For the Darcy problem, it is well known \cite{Brezzi-Fortin} that for
stable mixed finite element pairs, there exists an interpolant
$\Pi_{D,h}: \bV_D \cap H^1(\Omega_D) \to \bV_{D,h}$ such that for all $\bv_D \in H^1(\Omega_D)$,
\begin{subequations}\label{interp-Darcy}
  \begin{align}
    b_{D}(\Pi_{D,h}\bv_D, w_{D,h}) &= b_D(\bv_D,w_{D,h}), \quad
    \forall \,w_{D,h} \in W_{D,h}, \label{PiDh-div} \\
    \|\bv_D - \Pi_{D,h}\bv_D\|_{\Omega_D} &\lesssim h |\bv_D|_{1,\Omega_D},
    \label{PiDh-approx}\\
    \|\Pi_{D,h}\bv_D\|_{\Omega_D} &\lesssim \|\bv_D\|_{1,\Omega_D}.
    \label{PiDh-cont}
    \end{align}
\end{subequations}
Furthermore, the following continuity, coercivity,
and inf-sup condition hold:
\begin{subequations}
  \begin{align}
		a_{D}(\bu_{D,h},\bv_{D,h}) &\lesssim \|\bu_{D,h}\|_{\Omega_D}\|\bv_{D,h}\|_{\Omega_D}, \quad
    && \forall \, \bu_{D,h},\bv_{D,h} \in \bV_{D,h},
    \label{cont-a-Darcy}    \\
    a_{D}(\bv_{D,h},\bv_{D,h}) &\gtrsim \|\bv_{D,h}\|_{\Omega_D}^2, \quad
    && \forall \, \bv_{D,h} \in \bV_{D,h}, \label{coer-a-Darcy} \\
    b_{D}(\bv_{D,h},w_{D,h}) &\lesssim \|\bv_{D,h}\|_{\div;\Omega_D}\|w_{D,h}\|_{\Omega_D}, \quad
    && \forall \, \bv_{D,h} \in \bV_{D,h}, \, w_{D,h} \in W_{D,h}, \label{cont-b-Darcy} \\
    \sup_{\bv_{D,h} \in \bV_{D,h}\setminus 0}\frac{b_{D}(\bv_{D,h},w_{D,h})}{\|\bv_{D,h}\|_{\div;\Omega_D}}
    &\gtrsim \|w_{D,h}\|_{\Omega_D}, \quad
    && \forall w_{D,h} \in W_{D,h}.
    \label{inf-sup-Darcy}
  \end{align}
  \end{subequations}

We next discuss the choice of $\Lambda_h$. In order to simplify the presentation, we define
\begin{align}\label{L-Vn}
	\Lambda_h = \bV_{D,h}\cdot\bn|_{\Gamma},
\end{align}
which allows us to utilize the arguments from \cite{LSY}. \rev{With this choice, the following interface inf-sup condition holds \cite{ambartsumyan2019nonlinear}:
  \begin{equation}\label{inf-sup-b-gamma}
\inf_{\xi_h \in \Lambda_h \setminus 0} \ \sup_{\bv_h \in \bV_{h} \setminus 0}
\frac{b_{\Gamma}(\bv_h,\xi_h)}{\|\bv_{h}\|_V \| \xi_h \|_\Gamma} \gtrsim 1.
  \end{equation}
}
We note that a more general choice of $\Lambda_h$ is also possible. In particular, $\Lambda_h$
may consist of continuous or discontinuous polynomials of degree $m \ge 1$ on a
mesh $\Gamma_h$ different from the subdomain grids, satisfying for all
$\xi_h \in \Lambda_h$,
\begin{align}\label{L<Vn}
\|\xi_h\|_{\Gamma} \lesssim \|P_{D,h}\xi_h\|_{\Gamma},
\end{align}
where $P_{D,h}$ is the $L^2$-orthogonal projection onto
$\bV_{D,h}\cdot\bn|_{\Gamma}$.  For the treatment of this more general
choice, we refer the reader to \cite{GVY}, see also \cite{ACWY}. 

\rev{For the purpose of the analysis,} following \cite{LSY}, we consider \rev{a reduced} formulation of \eqref{weak-h} in the weakly continuous velocity space
\begin{align}\label{defn-Vh0}
	\bV_{h,c} := \{ \bv_h \in \bV_h: b_{\Gamma}(\bv_h,\xi_h) = 0 \,\,
	\forall \, \xi_h \in \Lambda_h\}.
\end{align}
The reduced problem is: find $(\bu_h, p_h) \in \bV_{h,c} \times W_h$ such that
\begin{subequations}\label{weak-h-0}
  \begin{align}
    a_h(\bu_h, \bv_h) + b_h(\bv_h, p_h) &=
    (\bm{f}_S, \rev{\cQ} \bv_{S,h})_{\Omega_S}, 
    & \forall \bv_h &\in \bV_{h,c},
    \label{weak-h-0-1} \\
    b_h(\bu_h,w_h) &= -(f_D, w_{D,h})_{\Omega_D}, 
    & \forall w_h &\in W_h.
    \label{weak-h-0-2}
  \end{align}
\end{subequations}

\rev{
\begin{lemma}\label{lem:equiv}
Method \eqref{weak-h} is equivalent to method \eqref{weak-h-0} in the following sense. For any solution $(\bu_h, p_h,\lambda_h)$ to \eqref{weak-h}, $(\bu_h, p_h)$ is a solution to \eqref{weak-h-0}. Conversely, for any solution $(\bu_h, p_h)$ to \eqref{weak-h-0}, there exists a unique $\lambda_h \in \Lambda_h$ such that $(\bu_h, p_h,\lambda_h)$ is a solution to \eqref{weak-h}.
\end{lemma}
}
\rev{
\begin{proof}
Let $(\bu_h, p_h,\lambda_h)$ be a solution to \eqref{weak-h}. Equation \eqref{weak-h-3} implies that $\bu_h \in \bV_{h,c}$. Taking $\bv_h \in \bV_{h,c}$ implies \eqref{weak-h-0-1}. Therefore $(\bu_h, p_h)$ is a solution to \eqref{weak-h-0}. Conversely, let $(\bu_h, p_h)$ be a solution to \eqref{weak-h-0}. Since $\bu_h \in \bV_{h,c}$, \eqref{weak-h-3} holds. Due to the inf-sup condition \eqref{inf-sup-b-gamma}, there exists a unique $\lambda_h \in \Lambda_h$ such that \eqref{weak-h-1} holds. Therefore $(\bu_h, p_h,\lambda_h)$ is a solution to \eqref{weak-h}.
\end{proof}
}

\begin{lemma}\label{interp-V-h0}
There exists an interpolant $\Pi_{h,c}: H^1(\Omega) \to \bV_{h,c}$ such that for all sufficiently smooth $\bv$,
  \begin{subequations}
    \begin{align}
& b_{h}(\Pi_{h,c}\bv, w_{h}) = b(\bv,w_{h}) \quad \forall \,w_{h} \in W_{h}, \label{Pih0-div} \\
& \|\Pi_{h,c}\bv\|_{V} \lesssim \|\bv\|_1, \label{Pih0-cont}\\
& \|\bv - \Pi_{h,c}\bv\|_{1,\Omega_S} + \|\bv - \Pi_{h,c}\bv\|_{\Omega_D}
\lesssim h
(\|\bv\|_1 + |\bv_S|_{2,\Omega_S}),
\label{Pih0-approx}\\
& \|\nabla\cdot(\bv - \Pi_{h,c}\bv)\|_{\Omega_D} \lesssim
h |\nabla\cdot \bv_D|_{1,\Omega_D}. \label{approx-div}
    \end{align}
    \end{subequations}
\end{lemma}
\begin{proof}
  The proof follows from the proofs of Lemma~4.3 and Proposition~4.2 in \cite{LSY},
  utilizing $\Pi_{S,h}$ from \eqref{interp-MAC} and $\Pi_{D,h}$ from \eqref{interp-Darcy}
  to build the interpolant in $\Omega_S$ and $\Omega_D$, respectively. In particular,
\begin{align}
  \Pi_{h,c}\bv|_{\Omega_S} = \Pi_{h,c}^S\bv_S = \Pi_{S,h}\bv_S, \qquad
  \Pi_{h,c}\bv|_{\Omega_D} = \Pi_{h,c}^D\bv_D = \Pi_{D,h}\bv_D + \delta_{D,h},
\end{align}
  where $\delta_{D,h} \in \bV_{D, h}$ is a suitably constructed correction that provides the weak continuity of the normal velocity. We omit further details.
\end{proof}

\begin{lemma}
The following inf-sup condition holds:
\begin{align}\label{inf-sup-0}
  \sup_{\bv_{h} \in \bV_{h,c}\setminus 0}\frac{b_{h}(\bv_{h},w_{h})}{\|\bv_{h}\|_V}
  \gtrsim \|w_{h}\|_W, \quad \forall w_{h} \in W_{h}.
  \end{align}
\end{lemma}
\begin{proof}
Let $w_h \in W_h$ be given. Since $w_h \in L^2_0(\Omega)$, it is known
\cite{Girault-Raviart} that there exists $\bv \in H^1_0(\Omega)$ such that
\begin{align}\label{div-problem}
	\nabla\cdot\bv &= -w_h \ \mbox{ in } \Omega, & 
	\|\bv\|_{1} &\lesssim \|w_h\|.
\end{align}
We then have, using \eqref{div-problem}, \eqref{Pih0-div}, and \eqref{Pih0-cont},
\begin{align*}
\|w_h\|_W \lesssim \frac{b(\bv,w_h)}{\|\bv\|_1} = \frac{b_h(\Pi_{h,c}\bv,w_h)}{\|\bv\|_1}
\lesssim \frac{b_h(\Pi_{h,c}\bv,w_h)}{\|\Pi_{h,c}\bv\|_V}.
\end{align*}
  \end{proof}

\begin{lemma}\label{lem:weak-h-0}
  Problem \eqref{weak-h-0} has a unique solution $(\bu_h, p_h) \in \bV_{h,c} \times W_h$ that satisfies
  \begin{align}
  	\|\bu_{S,h}\|_{1,\Omega_S} 
  	+ \|\bu_{D,h}\|_{\div;\Omega_D}
  	+ \|p_h\|_{\Omega} 
  	\lesssim
  	\| \bm{f}_S \|_{-1, \Omega_S}
  	+ \| f_D \|_{\Omega_D}.
  \end{align}
\end{lemma}
\begin{proof}
  \rev{Let $\bZ_{D,h} = \{\bv_{D,h} \in \bV_{D,h}: b_D(\bv_{D,h},w_{D,h}) = 0 \ \ \forall w_{D,h} \in W_{D,h}\}$ and
$$
\bZ_h = \bZ_{S,h} \times  \bZ_{D,h} = \{\bv_h \in \bV_h: b_h(\bv_h,w_h) = 0 \ \ \forall w_h \in W_h\}.
$$
From \eqref{coer-a-MAC} and \eqref{coer-a-Darcy}, using \eqref{eq: div property}, we obtain
  \begin{align}\label{coer-a}
    a_h(\bv_h,\bv_h) &\gtrsim \|\bv_{S,h}\|_{1,\Omega_S}^2 + \|\bv_{D,h}\|_{\div;\Omega_D}^2
    \quad \forall \, \bv_h \in \bZ_h.
  \end{align}
}  
The assertion of the lemma follows from \eqref{coer-a} 
and the inf-sup condition \eqref{inf-sup-0}, using the general theory
of saddle point problems \cite{Brezzi-Fortin}.
\end{proof}

\rev{Lemma~\ref{lem:weak-h-0} and Lemma~\ref{lem:equiv} imply well posedness of the mortar MFE method \eqref{weak-h}.}
\rev{  
\begin{lemma}
  Problem \eqref{weak-h} has a unique solution $(\bu_h, p_h, \lambda_h) \in \bV_h \times W_h \times \Lambda_h$ that satisfies
  \begin{align}
  	\|\bu_{S,h}\|_{1,\Omega_S} 
  	+ \|\bu_{D,h}\|_{\div;\Omega_D}
  	+ \|p_h\|_{\Omega} + \|\lambda_h\|_{\Gamma}
  	\lesssim
  	\| \bm{f}_S \|_{\Omega_S}
  	+ \| f_D \|_{\Omega_D}.
  \end{align}
\end{lemma}
}
\section{Error estimates}
\label{sec:error}

In this section we establish convergence rates for the mortar finite element solution to the coupled Stokes-Darcy problem.

\begin{theorem}\label{thm-conv}
  Assuming sufficiently smooth solution to \eqref{weak}, the solution $(\bu_h,p_h)$ of the
  mortar finite element method \eqref{weak-h-0} satisfies
  \begin{subequations}
  \begin{align}
    & \|\bu - \bu_h\|_V \lesssim h(\|\bu\|_1 + |\bu_S|_{2,\Omega_S}
    + |\nabla\cdot \bu_D|_{1,\Omega_D} + |p_S|_{1,\Omega_S} +
    |\lambda|_{1,\Gamma}
    \rev{+ \| \bm{f}_S \|_{\Omega_S}}),   \label{err-est}\\
    & \|p - p_h\|_W \lesssim h(\|\bu\|_1 + |\bu_S|_{2,\Omega_S}
    + |p_S|_{1,\Omega_S} + |p_D|_{1,\Omega_D} +
    |\lambda|_{1,\Gamma}
    \rev{+ \| \bm{f}_S \|_{\Omega_S}}).   \label{err-est-p}
  \end{align}
  \end{subequations}
\end{theorem}
\begin{proof}
  Let $Q_h = (Q_{S,h},Q_{D,h})$, where $Q_{i,h}$ is the $L^2$-orthogonal projection onto $W_{i,h}$, $i = S,D$. The two operators satisfy, for all $w_S \in H^1(\Omega_S)$ and $w_D \in H^1(\Omega_D)$, 
  \begin{align}\label{approx-Q}
    \|w_S - Q_{S,h} w_S\|_{\Omega_S} &\lesssim h |w_S|_{1,\Omega_S}, &
    \|w_D - Q_{D,h} w_D\|_{\Omega_D} &\lesssim h |w_D|_{1,\Omega_D}.
  \end{align}
We start by noting that $\bV_h \times W_h \subset \bV \times W$. Thus, subtracting \eqref{weak-h-0-1}--\eqref{weak-h-0-2} from \eqref{weak-1}--\eqref{weak-2} leads us to the error equations:
\begin{subequations} \label{eq: proof1}
\begin{align}
		a(\bu, \bv_h) 
		- a_h(\bu_h, \bv_h) 
		+ b(\bv_h, p)
		- b_h(\bv_h, p_h) 
		+ b_{\Gamma}(\bv_h,\lambda)
		&=
    \rev{(\bm{f}_S, (I - \cQ) \bv_{S,h})_{\Omega_S},} 
    & \forall \bv_h &\in \bV_{h,c}, \label{eq: proof1a} \\
    b(\bu,w_h) - b_h(\bu_h,w_h) &= 0, & \forall w_h &\in W_h. \label{eq: proof1b}
\end{align}
\end{subequations}

Using $\nabla\cdot \bV_{D,h} = W_{D,h}$, cf. \eqref{eq: div property}, we proceed by considering the following differences:
\begin{subequations} \label{eq: discrepancy terms}
\begin{align}
		a_h(\Pi_{h,c} \bu, \bv_h) - 
		a(\bu, \bv_h) &= 
 		a_D(\Pi_{h,c}^D\bu_D - \bu_D,\bv_{D,h})
		+ 
		\left(
 		a_{S,h}(\Pi_{S,h}\bu_S,\bv_{S,h}) - a_S(\bu_S,\bv_{S,h})
		\right) \nonumber \\
 		&=: 
 		R_{u, D}(\bu_D,\bv_{D,h})
 		+ 
 		R_{u, S}(\bu_S,\bv_{S,h}) \\
		b_h(\bv_h, Q_h p)
		- b(\bv_h, p)
		&= b_{S,h}(\bv_{S,h},Q_{S,h}p_S) - b_S(\bv_{S,h},p_S) \nonumber\\
		&=:
		R_{p,S}(p_S,\bv_{S,h})
\end{align}
\end{subequations}
Adding \eqref{eq: discrepancy terms} to \eqref{eq: proof1a} and using property \eqref{Pih0-div} of $\Pi_{h,c}$ in \eqref{eq: proof1b}, we rewrite \eqref{eq: proof1} as
\begin{subequations} \label{eq: err-eqs}
\begin{align}
  a_h(\Pi_{h,c}\bu -\bu_h, \bv_h) + &b_h(\bv_h, Q_h p - p_h)
    \nonumber \\
    &=
    - b_{\Gamma}(\bv_h,\lambda)
    \rev{+ (\bm{f}_S, (I - \cQ) \bv_{S,h})_{\Omega_S}} \nonumber\\
    & \qquad + R_{u, D}(\bu_D,\bv_{D,h}) +
    R_{u, S}(\bu_S,\bv_{S,h}) + R_{p,S}(p_S,\bv_{S,h}), & 
    \forall \bv_h &\in \bV_{h,c},
    \label{err-eq-1}\\
    b_h(\Pi_{h,c}\bu - \bu_h,w_h) &= 0, & 
    \forall w_h &\in W_h, \label{err-eq--2}
\end{align}
\end{subequations}

We now take $\bv_h = \Pi_{h,c}\bu -\bu_h$ and $w_h = p_h - Q_h p$. \rev{Note that \eqref{err-eq--2} implies that $\Pi_{h,c} \bu -\bu_h \in \bZ_h$.} By summing the equations \eqref{eq: err-eqs} and using the coercivity \eqref{coer-a} we derive:
\begin{align}\label{err-bound-1}
  \|\Pi_{h,c}\bu -\bu_h\|_V^2 & \lesssim
  |b_{\Gamma}(\bv_h,\lambda)|
  \rev{+ |(\bm{f}_S, (I - \cQ) \bv_{S,h})_{\Omega_S}|} \nonumber\\
  & \qquad + |R_{u, D}(\bu_D,\bv_{D, h})|
  + |R_{u, S}(\bu_S,\bv_{S, h})|
  + |R_{p,S}(p_S,\bv_{S, h})|.
\end{align}

We proceed by bounding the \rev{five} terms on the right-hand side. 
The first term is the non-conforming
error on the interface. Using the definition \eqref{defn-Vh0} of $\bV_{h,c}$ and the fact that $\Lambda_h = \bV_{D,h}\cdot\bn|_{\Gamma}$, we have
\begin{align}\label{b-Gamma-bound}
  |b_{\Gamma}(\Pi_{h,c}\bu -\bu_h,\lambda)| & =
  |b_{\Gamma}(\Pi_{h,c}\bu -\bu_h,\lambda - P_{\Lambda_h}\lambda)|
  = |\lang \Pi_{S,h}\bu_S - \bu_{S,h},\lambda - P_{\Lambda_h}\lambda\rang_{\Gamma}| \nonumber \\
  & \lesssim h \|\Pi_{S,h}\bu_S - \bu_{S,h}\|_{1,\Omega_S}|\lambda|_{1,\Gamma}
\le h \|\Pi_{h,c}\bu -\bu_h\|_V |\lambda|_{1,\Gamma},
\end{align}
where $P_{\Lambda_h}$ is the $L^2$-orthogonal projection onto $\Lambda_h$, and we
used its approximation property
\begin{equation}\label{approx-PLambda}
\forall \xi \in H^1(\Gamma), \quad \|\xi - P_{\Lambda_h}\xi\|_\Gamma \lesssim h |\xi|_{1,\Gamma},
\end{equation}
and the trace inequality
$$
\forall \bv_S \in (H^1(\Omega_S))^n, \quad \|\bv_S\|_\Gamma \lesssim \|\bv_S\|_{1,\Omega_S}.
$$

\rev{
We bound the second term by using the Cauchy-Schwarz inequality and the approximation properties of the interpolant $\cQ$:
\begin{align}\label{RuD}
  |(\bm{f}_S, (I - \cQ) \bv_{S,h})_{\Omega_S}|
  &\lesssim \| \bm{f}_S \|_{\Omega_S} \| (I - \cQ) \bv_{S,h} \|_{\Omega_S}
  \lesssim h \| \bm{f}_S \|_{\Omega_S} \| \bv_{S,h} \|_{1, \Omega_S}
\end{align}
}

The \rev{third} term is bounded by the continuity of $a_D$ \eqref{cont-a-Darcy} and the approximation property \eqref{Pih0-approx}:
\begin{align}
  R_{u, D}(\bu_D,\bv_{D,h}) \lesssim
  \|\Pi_{h,c}^D \bu_D - \bu_D\|_{\Omega_D}\|\bv_{D,h}\|_{\Omega_D}
  \lesssim h \, \|\bu\|_1 \|\bv_{D,h}\|_{\Omega_D}.
\end{align}
\rev{To bound the final two terms in \eqref{err-bound-1}, we first note that, due to \eqref{PiSh-div}, $\Pi_{S,h}\bu_S \in \bZ_{S,h}$, implying that $a_{S,h}(\Pi_{S,h}\bu_S,\bv_{S,h})$ can be expressed as in \eqref{a-sh-new}. In addition, since $\nabla \cdot\bu_S = 0$, the elliptic term in $a_S(\bu_S,\bv_S)$ can be expressed as $(\mu\nabla\bu_S,\nabla\bv_S)_{\Omega_S}$.} Therefore, following the argument in \cite{Han-Wu}, we have
\begin{align}\label{R1R2}
  |R_{u, S}(\bu_S,\bv_{S,h})| &\lesssim h |\bu_S|_{2,\Omega_S}\|\bv_{S,h}\|_{1,\Omega_S}, &
  |R_{p,S}(p_S,\bv_{S,h})| &\lesssim h |p_s|_{1,\Omega_S}\|\bv_{S,h}\|_{1,\Omega_S}.
\end{align}
Combining \eqref{err-bound-1}--\eqref{b-Gamma-bound} and \rev{\eqref{RuD}}--\eqref{R1R2}, we obtain
\begin{align} \label{eq: final error eq u}
  \|\Pi_{h,c}\bu -\bu_h\|_V \lesssim h(\|\bu\|_1 +  |\bu_S|_{2,\Omega_S} + |p_S|_{1,\Omega_S} + |\lambda|_{1,\Gamma} \rev{+ \| \bm{f}_S \|_{\Omega_S}}).
\end{align}
The bound on $\|\bu - \bu_h\|_V$ in \eqref{err-est} now follows from
\eqref{Pih0-approx} and \eqref{approx-div}.  To bound $\|p - p_h\|_W$, we use the inf-sup condition \eqref{inf-sup-0} and the error equation \eqref{err-eq-1}:
\begin{align*}
  \|Q_h p - p_h\|_W  & \lesssim
  \sup_{\bv_{h} \in \bV_{h,c}\setminus 0}\|\bv_{h}\|_V^{-1} \, b_{h}(\bv_{h},Q_h p - p_h) \\
  & = \sup_{\bv_{h} \in \bV_{h,c}\setminus 0}\|\bv_{h}\|_V^{-1}
  (- a_h(\Pi_{h,c}\bu -\bu_h, \bv_h)
  - b_{\Gamma}(\bv_h,\lambda - P_{\Lambda_h}\lambda)
  \rev{+ (\bm{f}_S, (I - \cQ) \bv_{S,h})_{\Omega_S}} \\
  & \qquad\qquad\qquad\qquad\quad + R_{u, D}(\bu_D,\bv_{D,h}) + R_{u, S}(\bu_S,\bv_{S,h}) + R_{p,S}(p_S,\bv_{S,h}))\\
  & \lesssim h \, (\|\bu\|_1 + |\bu_S|_{2,\Omega_S} + |p_S|_{1,\Omega_S} + |\lambda|_{1,\Gamma} \rev{+ \| \bm{f}_S \|_{\Omega_S}}),
\end{align*}
where we used bounds \eqref{approx-PLambda}--\eqref{eq: final error eq u} in the last inequality. The bound on $\|p - p_h\|_W$ in \eqref{err-est-p} now follows from the approximation property \eqref{approx-Q} and the triangle inequality.
\end{proof}

\begin{theorem}\label{thm-conv-mortar}
If the solution to \eqref{weak} is sufficiently smooth, then the mortar variable $\lambda_h \in \Lambda_h$ satisfies
	\begin{align}
		\| \lambda - \lambda_h \|_\Gamma \lesssim 
		h(\|\bu\|_1 + |\bu_S|_{2,\Omega_S} + |p_S|_{1,\Omega_S} +
    |\lambda|_{1,\Gamma}
    \rev{+ \| \bm{f}_S \|_{\Omega_S}}).
	\end{align}
\end{theorem}
\begin{proof}
  We start by considering the error equation obtained by subtracting \eqref{weak-h-1} from \eqref{weak-1} and testing with $\bv_h = (0,\bv_{D,h})$:
\begin{align} \label{eq: error for lambda}
    a_D(\bu_D - \bu_{D,h}, \bv_{D,h}) + b_D(\bv_{D,h}, p_D - p_{D,h})
    + \lang \bv_{D,h}\cdot\bn_D,\lambda - \lambda_h \rang_{\Gamma} =
    0, \quad \forall \bv_{D,h} \in \bV_{D,h}.
\end{align}
The proof then relies on choosing an appropriate test function $\bv_h$. We recall the inf-sup condition \eqref{inf-sup-b-gamma}. In particular, it is shown in the proof of \cite[Lemma 5.1]{ambartsumyan2019nonlinear} that, for given $\xi_h \in \Lambda_h$, there exists $\bv_{D, h}^\xi \in \bV_{D, h}$ that satisfies
\begin{align} \label{eq: props v_Dh}
\nabla \cdot \bv_{D, h}^\xi &= 0, &
\lang \bv_{D, h}^\xi \cdot \bn_D, \xi_h \rang_{\Gamma} &= \| \xi_h \|_\Gamma^2, &
  \|\bv_{D,h}^\xi\|_{\Omega_D} &\lesssim \| \xi_h \|_\Gamma.
\end{align}
We now set $\xi_h = P_{\Lambda_h} \lambda - \lambda_h$ and choose the test function $\bv_{D, h} = \bv_{D, h}^\xi$ in \eqref{eq: error for lambda}. Using the properties \eqref{eq: props v_Dh} and the choice \eqref{L-Vn}, we derive
\begin{align*}
\| P_{\Lambda_h} \lambda - \lambda_h \|_\Gamma^2 &=
\< \bv^\xi_{D,h}\cdot\bn_D, P_{\Lambda_h} \lambda - \lambda_h \>_{\Gamma} =
\lang \bv^\xi_{D,h}\cdot\bn_D, \lambda - \lambda_h \rang_{\Gamma} =
a_D(\bu_{D, h} - \bu_D, \bv^\xi_{D, h}) \\
& \lesssim
\| \bu_{D, h} - \bu_D \|_{\Omega_D} \| \bv^\xi_{D, h} \|_{\Omega_D} 
\lesssim \| \bu_{D, h} - \bu_D \|_{\Omega_D} \| P_{\Lambda_h} \lambda - \lambda_h \|_\Gamma.
\end{align*}
To conclude the proof, we invoke the bound \eqref{err-est} restricted to
$\| \bu_{D, h} - \bu_D \|_{\Omega_D}$, the approximation property \eqref{approx-PLambda}, and the triangle inequality.
\end{proof}

\section{Domain decomposition algorithm}
\label{sec:DD}

In this section we describe a non-overlapping domain decomposition
for the solution of the algebraic system resulting from
\eqref{weak-h}. The algorithm reduces \eqref{weak-h} to
solving an interface problem for $\lambda_h$ and requires only
decoupled Stokes and Darcy subdomain solves. Following
\cite{VasWangYot}, we consider two sets of complementary subdomain
problems. Given $\lambda_h \in \Lambda_h$, let
$(u_{i,h}^*(\lambda_h),p_{i,h}^*(\lambda_h)) \in \bV_{i,h}\times W_{i,h}$, $i = S,D$, be the solution of Stokes or Darcy subdomain problems with specified normal stress (for Stokes) or pressure (for Darcy) boundary condition $\lambda_h$ on $\Gamma$:
\begin{subequations}\label{weak-h-*}
  \begin{align}
    a_{i,h}(\bu^*_{i,h}(\lambda_h), \bv_{i,h}) + b_{i,h}(\bv_{i,h}, p^*_{i,h}(\lambda_h))
    + \lang \lambda_h, \bv_{i,h} \cdot\bn_i \rang_{\Gamma} &=
    0, &
    \forall \bv_{i,h} &\in \bV_{i,h},
    \label{weak-h-*-1} \\
    b_{i,h}(\bu^*_{i,h}(\lambda_h),w_{i,h}) &= 0, &
    \forall w_{i,h} &\in W_{i,h},
    \label{weak-h-*-2}
  \end{align}
\end{subequations}
where we set $a_{D,h}(\cdot,\cdot) = a_{D}(\cdot,\cdot)$ and
$b_{D,h}(\cdot,\cdot) = b_{D}(\cdot,\cdot)$, which allows us to unify the notation
for the two types of problems. We also consider the set of complementary subdomain
problems for $(\bar \bu_{i,h},\bar p_{i,h}) \in \bV_{i,h}\times W_{i,h}$, $i = S,D$, such that
\begin{subequations}\label{weak-h-bar}
  \begin{align}
    a_{i,h}(\bar \bu_{i,h}, \bv_{i,h}) + b_{i,h}(\bv_{i,h}, \bar p_{i,h})
    &= (\bm{f}_S, \rev{\cQ} \bv_{S,h})_{\Omega_S}, &
    \forall \bv_{i,h} &\in \bV_{i,h},
    \label{weak-h-bar-1} \\
    b_{i,h}(\bar \bu_{i,h},w_{i,h}) &= -(f_D,w_{D,h})_{\Omega_D},    &
    \forall w_{i,h} &\in W_{i,h}.
    \label{weak-h-bar-2}
  \end{align}
\end{subequations}
The first set of subdomain problems incorporates interface data as
boundary condition, while setting the outside boundary conditions and
source terms to zero. The second set has zero data on the interface
and uses the true outside boundary conditions and source terms. It is
easy to check that the solution to \eqref{weak-h} satisfies
\begin{align*}
	\bu_h &= \bu^*_{h}(\lambda_h) + \bar \bu_h, &
	p_h &= p_h^*(\lambda_h) + \bar p_h,
\end{align*}
where $\lambda_h \in \Lambda_h$ is the solution of the interface problem
\begin{align}\label{interface-pressure}
  s_h(\lambda_h,\xi_h) &\equiv - b_{\Gamma}(\bu^*_{h}(\lambda_h),\xi_h)
  = b_{\Gamma}(\bar \bu_h,\xi_h), &
  \forall \xi_h &\in \Lambda_h.
\end{align}

\begin{lemma}
  The bilinear form $s_h(\lambda_h,\xi_h)$ is symmetric and positive definite on
  $\Lambda_h$.
\end{lemma}
\begin{proof}
  The proof is similar to the proof of Lemma~5.1 in
  \cite{VasWangYot}. We provide it here for completeness.
  Taking $\bv_{i,h} = \bu^*_{i,h}(\xi_h)$, $i=S,D$ in \eqref{weak-h-*} and summing
implies that
\begin{align*}
s_h(\xi_h,\lambda_h) & = - \lang \lambda_h, \bu^*_{S,h}(\xi_h) \cdot\bn_S \rang_{\Gamma}
- \lang \lambda_h, \bu^*_{D,h}(\xi_h) \cdot\bn_D \rang_{\Gamma} \\
& = a_{S,h}(\bu^*_{S,h}(\lambda_h),\bu^*_{S,h}(\xi_h)) + b_{S,h}(\bu^*_{S,h}(\xi_h), p^*_{S,h}(\lambda_h)\\
& \quad + a_{D}(\bu^*_{D,h}(\lambda_h),\bu^*_{D,h}(\xi_h)) + b_{D,h}(\bu^*_{D,h}(\xi_h), p^*_{D,h}(\lambda_h) \\
& = a_{S,h}(\bu^*_{S,h}(\lambda_h),\bu^*_{S,h}(\xi_h)) +
a_{D}(\bu^*_{D,h}(\lambda_h),\bu^*_{D,h}(\xi_h)),
\end{align*}
which implies that $s_h(\cdot,\cdot)$ is symmetric and positive
semi-definite, using the coercivity \eqref{coer-a-MAC} of
$a_{S,h}(\cdot,\cdot)$ and \eqref{coer-a-Darcy} of
$a_D(\cdot,\cdot)$. Due to the zero outside boundary conditions and
source terms in \eqref{weak-h-*}, it is clear that
$\bu^*_{i,h}(\lambda_h) = 0$ if and only if $\lambda_h = 0$, which implies
that $s_h(\cdot,\cdot)$ is positive definite.
\end{proof}

As a consequence of the above lemma, the conjugate gradient (CG) algorithm
can be applied for solving the interface problem
\eqref{interface-pressure}. Each CG iteration requires evaluating
$s_h(\lambda_h,\xi_h)$, which involves solving decoupled Stokes and Darcy
subdomain problems \eqref{weak-h-*}.

\subsection{Implementation}
We next describe how the above algorithm is implemented when using the MAC scheme \eqref{MAC-momentum-1}--\eqref{MAC-mass}. The term $\lang \lambda_h, \bv_{i,h} \cdot\bn_i \rang_{\Gamma}$ in \eqref{weak-h-*-1} that incorporates the mortar data as boundary condition for the subdomain solves can be written as
\begin{align*}
\lang \lambda_h, \bv_{i,h} \cdot\bn_i \rang_{\Gamma}
= \lang P_{i,h}\lambda_h, \bv_{i,h} \cdot\bn_i \rang_{\Gamma},
\end{align*}
where $P_{i,h}$ is the $L^2$-orthogonal projection onto
$\bV_{i,h}\cdot\bn|_{\Gamma}$. On the Darcy side, due to the mortar choice $\Lambda_h = \bV_{D,h}\cdot\bn|_{\Gamma}$, cf. \eqref{L-Vn}, the mortar data is already in the correct space. On the Stokes side, it needs to be projected first into $\bV_{S,h}\cdot\bn|_{\Gamma}$ before using it as a normal stress boundary data for the Stokes solve. In the context of the MAC scheme \eqref{MAC-momentum-1}--\eqref{MAC-mass}, $\bV_{S,h}\cdot\bn|_{\Gamma}$ consists of piecewise constant functions on the trace of the primal grid on $\Gamma$.

\section{Numerical results}
\label{sec:numerical}

In this section, we investigate the performance and applicability of the proposed method through the use of three numerical test cases in two dimensions. Case 1 investigates the convergence of the method predicted in Section~\ref{sec:error} by employing a known analytical solution. Case 2 is more challenging and considers flow in a channel past a porous obstacle. Finally, we illustrate the flexibility of the method by considering regional mesh refinements in Case 3.

\subsection{Case 1: Convergence test}

To test the convergence of the method, we use the following analytical solution (cf. \cite{vassilev-2009}):
\begin{subequations}
\begin{align}
    \bu_S &=
        \begin{pmatrix}
            \left( 2 - x \right) \left( 1.5 - y \right) \left( y - \beta \right) + G \omega \cos \left( \omega x \right) \\
            - \frac{y^3}{3} + \frac{y^2}{2} \left( \beta + 1.5 \right) - 1.5 \beta y - 0.5 + \sin \left( \omega x \right)
        \end{pmatrix} &
    \bu_D &=
        \begin{pmatrix}
            \omega \cos \left( \omega x \right) y \\
            \chi \left( y + 0.5 \right) + \sin \left( \omega x \right)
        \end{pmatrix} \\
    p_S &= -\frac{\sin \left( \omega x \right) + \chi}{2K} + 2 \mu \left( 0.5 - \beta \right) + \cos \left( \pi y \right) &
    p_D &= -\frac{\chi}{K} \frac{\left( y + 0.5 \right)^2}{2} - \frac{\sin \left( \omega x \right) y}{K},
\end{align}
\end{subequations}
where
$$
\mu = 1,\enskip K=1,\quad \alpha=0.5,\quad G = \frac{\sqrt{\mu
K}}{\alpha}, \quad \omega=6, \quad
\beta = \frac{1-G}{2(1+G)}, \quad \chi = \frac{-30\beta-17}{48}.
$$

The computational domain is taken to be $\overline\Omega=\overline\Omega_S\cup\overline\Omega_D$, where
$\Omega_S=(0,1)\times(\frac{1}{2},1)$ and
$\Omega_D=(0,1)\times(0,\frac{1}{2})$. Dirichlet boundary conditions based on the analytical solutions for $\bu_S$ and $p_D$ are used on all outer boundaries.
We start with a $15 \times 15$ square grid in $\Omega_D$ and a $16 \times 16$ square grid in the $\Omega_S$. We consider two choices for the mortar space on the interface: piecewise-constant satisfying $\Lambda_h = \bV_{D,h}\cdot\bn|_{\Gamma}$, cf. \eqref{L-Vn}, with $15$ mortar elements, and continuous piecewise-linear with $14$ mortar elements, which satisfies \eqref{L<Vn}. This grid is then refined $5$ times, and the measured errors and convergence rates are listed in Tables~\ref{tab:convtest}--\ref{tab:convtest-P1mortar-superconv}. The error norms are computed as follows. Consider the $L^2(\Omega_i)$-norm
\begin{equation*}
  \|\varphi\|_{i} = \left(\sum_{E \in \Omega_{i, h}} \int_E \varphi^2\right)^{1/2},
  \quad i \in {S,D}.
\end{equation*}
The pressure $p_{D,h}$ is a piecewise constant function and $p_{S,h}$ is reconstructed as a piecewise constant function based on its degrees of freedom at the cell-centers. The pressure errors $e_{p,i}$ are measured in the above norm:
\begin{equation*}
\quad e_{p,i} = \|p_i - p_{i,h}\|_{i}, \quad i \in {S,D}.
\end{equation*}
For the $L^2$-norms of $\bu_S$ and $\bu_D$, the following edge-norm is employed:
\begin{equation*}
    \|\bv_i\|_{e,i} = \left(\sum_{E \in \Omega_{i, h}} 
      | E |
      \sum_{e \subset \partial E} \frac{1}{|e|}
      \int_e  (\bv_i \cdot \bn)^2 \right)^{1/2},
\label{eq:velocity_norm}
\end{equation*}
in which each $e$ is an edge of the mesh. We take $e_{u_D} = \|\bu_D - \bu_{D,h}\|_{e,D}$. We note that for the discrete vector $\bu_{D,h} \in \bV_{D,h}$, $\bu_{D,h}\cdot\bn$ is constant on each edge. For $\bu_S$ we use the following $H^1(\Omega_S)$-type norm:
\begin{equation*}
  \|\bv_S\|_{S} = \left(\|\bv_S\|_{e,S}^2 + \left\|\frac{\d v_{S,1}}{\d x}\right\|_{S}^2 + \left\|\frac{\d v_{S,2}}{\d y}\right\|_{S}^2
  + \left\|\frac{\d v_{S,1}}{\d y}\right\|_{S}^2
  + \left\|\frac{\d v_{S,2}}{\d x}\right\|_{S}^2
    \right)^{1/2}, \quad e_{u_S} = \|\bu_S - \bu_{S,h}\|_{S}.
\end{equation*}
In the the first term on the right hand side above, $\bu_{S,h}\cdot\bn$ is reconstructed as constant on each edge, based on the MAC normal velocity degrees of freedom et the edge midpoints. In the second and third terms,
$\frac{\d u_{S,h,1}}{\d x}$ and $\frac{\d u_{S,h,2}}{\d y}$ are reconstructed as constants on each primary element $E$ based on their values at the cell-center $\cC$ computed in \eqref{MAC-vel-1}. In the last two terms, $\frac{\d u_{S,h,1}}{\d y}$ and $\frac{\d u_{S,h,2}}{\d x}$ are reconstructed as bilinear functions on each primary element $E$ based on their values at the vertices $\cV$ computed in \eqref{MAC-vel-2}. In Tables \ref{tab:convtest} and \ref{tab:convtest-P1mortar} we report the errors and convergence rates with piecewise-constant and piece-linear mortars, respectively. We observe first order convergence for all subdomain variables, as predicted by Theorem~\ref{thm-conv}. For the mortar variable we observe first order convergence for the piecewise-constant choice, which is consistent with Theorem~\ref{thm-conv-mortar}, and second order convergence for the piecewise-linear case. The latter is not covered by the presented theory, but it is consistent with the approximation properties of the mortar space.

We also report the errors and convergence rates using superconvergent norms based on computing the error integrals on the elements $E$ and edges $e$ with the midpoint quadrature rule, see Table~\ref{tab:convtest-superconv} for piecewise-constant mortars and Table~\ref{tab:convtest-P1mortar-superconv} for piecewise-linear mortars. We observe second order convergence for all variables. While the superconvergence analysis is beyond the scope of this paper, the rates are consistent with known superconvergence for the MAC scheme for Stokes \cite{Li-Sun,Li-Rui} and the RT$_0$ MFE method for Darcy \cite{ACWY}. Interestingly, to the best of our knowledge, these are the first numerical results in the literature reporting second order convergence for the MAC velocity in the $H^1$-norm. 

\begin{figure}
    \centering
    \begin{subfigure}{0.32\textwidth}
        \includegraphics[width=\textwidth]{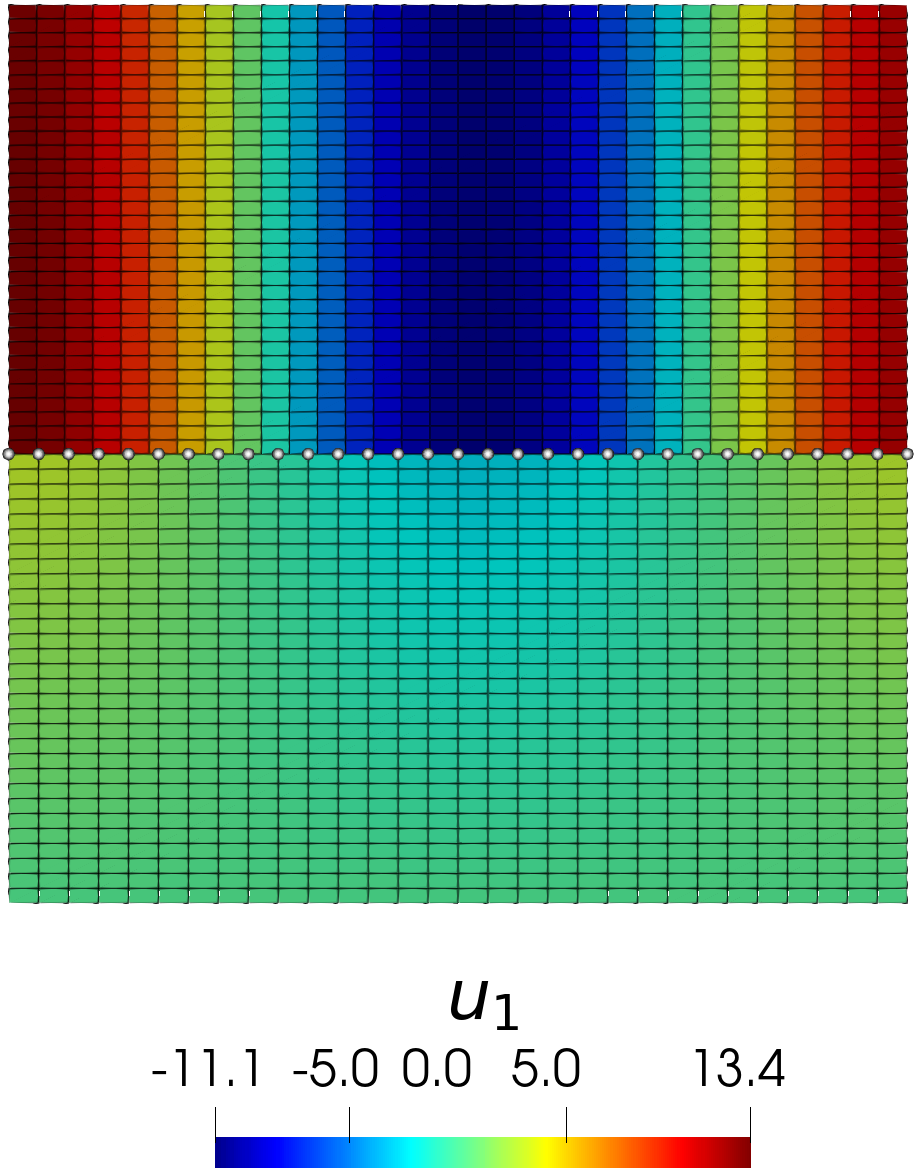}
        \caption{}
        \label{fig:convtest_vx}
    \end{subfigure}
    \begin{subfigure}{0.32\textwidth}
        \includegraphics[width=\textwidth]{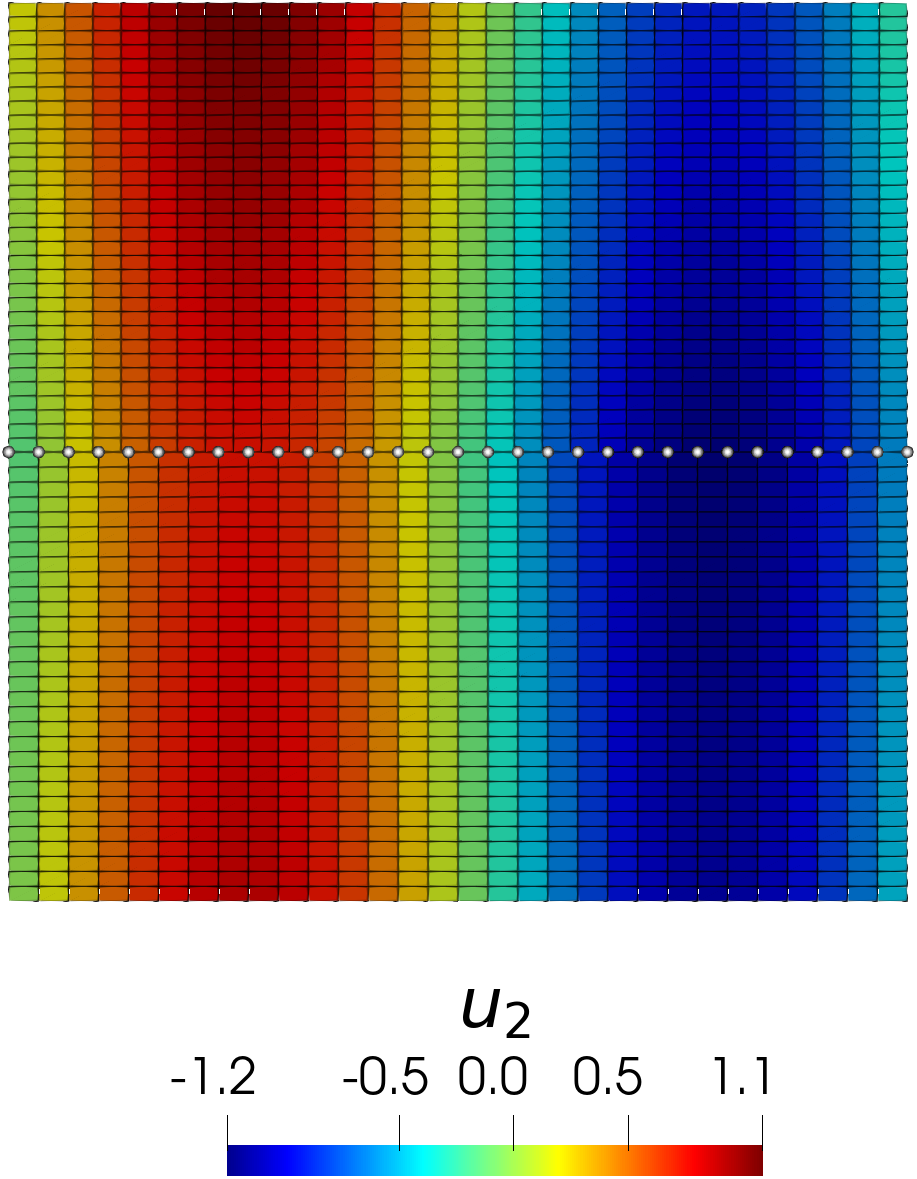}
        \caption{}
        \label{fig:convtest_vy}
    \end{subfigure}
    \begin{subfigure}{0.32\textwidth}
        \includegraphics[width=\textwidth]{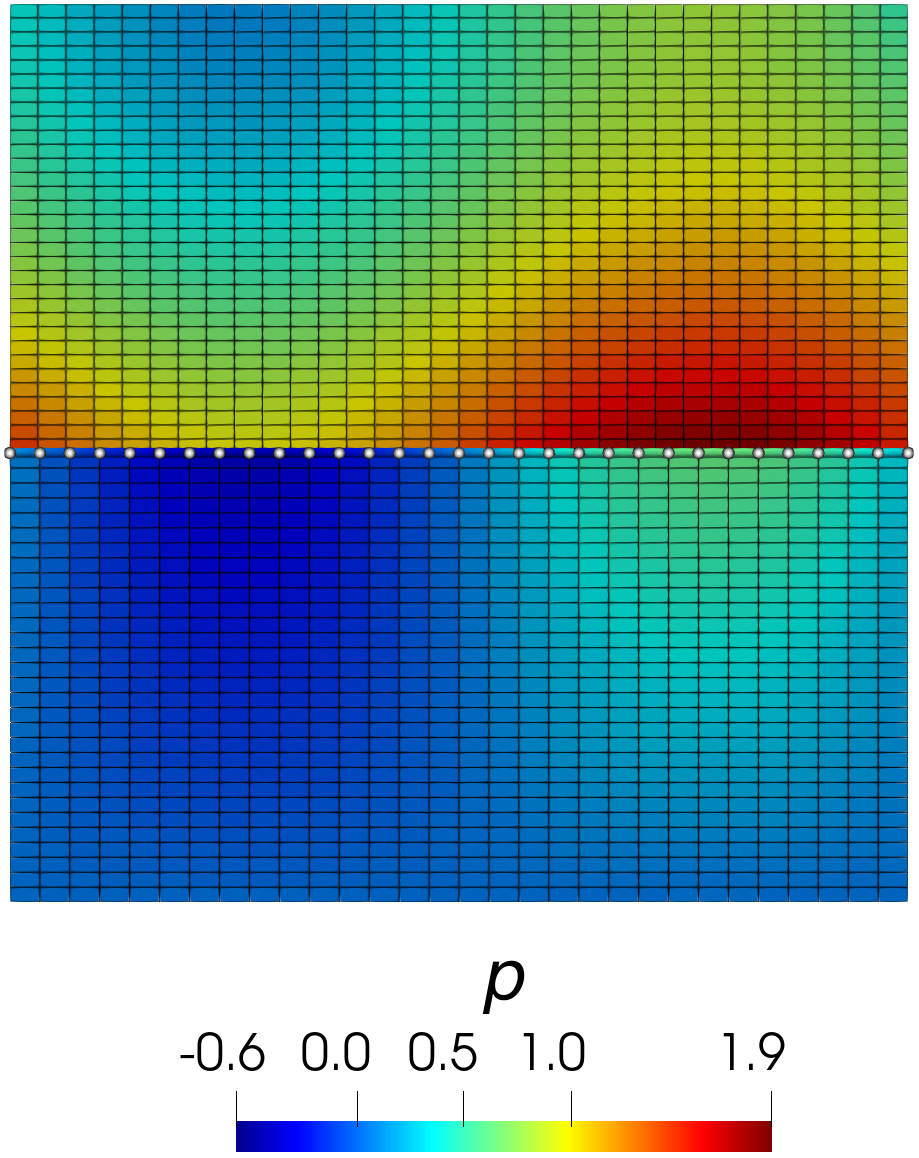}
        \caption{}
        \label{fig:convtest_p}
    \end{subfigure}
    \caption{Velocity and pressure distributions after the first refinement for Case 1.}
    \label{fig:convtest}
\end{figure}

\begin{table}[h]
    {\footnotesize
    \caption{Errors and convergence rates using piecewise-constant mortars for Case 1.}
    \label{tab:convtest}
    \begin{center}
    \begin{tabular}{l |l l|l l|l l|l l|l l}
        \toprule
        $i$ & $e_{p, D}$ & $r_{p, D}$ & $e_{u, D}$ & $r_{u, D}$ & $e_{p, S}$ & $r_{p, S}$ & $e_{u, S}$ & $r_{u, S}$ & $e_{\lambda}$ & $r_{\lambda}$ \\
        \midrule
        0 & 1.70e-02 &  & 9.21e-02 &  & 2.74e-01 &  & 3.80e+00 &  & 3.99e-02 &  \\
        1 & 8.53e-03 & 9.98e-01 & 4.49e-02 & 1.04e+00 & 7.01e-02 & 1.96e+00 & 1.90e+00 & 9.98e-01 & 1.99e-02 & 1.00e+00 \\
        2 & 4.26e-03 & 9.99e-01 & 2.24e-02 & 1.01e+00 & 1.88e-02 & 1.90e+00 & 9.52e-01 & 1.00e+00 & 9.98e-03 & 1.00e+00 \\
        3 & 2.13e-03 & 1.00e+00 & 1.12e-02 & 1.00e+00 & 5.71e-03 & 1.72e+00 & 4.76e-01 & 1.00e+00 & 4.99e-03 & 1.00e+00 \\
        4 & 1.07e-03 & 1.00e+00 & 5.58e-03 & 1.00e+00 & 2.16e-03 & 1.40e+00 & 2.38e-01 & 1.00e+00 & 2.49e-03 & 1.00e+00 \\
        5 & 5.33e-04 & 1.00e+00 & 2.79e-03 & 1.00e+00 & 9.76e-04 & 1.15e+00 & 1.19e-01 & 1.00e+00 & 1.25e-03 & 1.00e+00 \\
        \bottomrule
    \end{tabular}
    \end{center}
    }
\end{table}

\begin{table}[h]
    {\footnotesize
    \caption{Errors and convergence rates using piecewise-linear mortars for Case 1.}
    \label{tab:convtest-P1mortar}
    \begin{center}
    \begin{tabular}{l |l l|l l|l l|l l|l l}
        \toprule
        $i$ & $e_{p, D}$ & $r_{p, D}$ & $e_{u, D}$ & $r_{u, D}$ & $e_{p, S}$ & $r_{p, S}$ & $e_{u, S}$ & $r_{u, S}$ & $e_{\lambda}$ & $r_{\lambda}$ \\
        \midrule
        0 & 1.70e-02 &  & 9.11e-02 &  & 2.72e-01 &  & 3.80e+00 &  & 1.84e-03 &  \\
        1 & 8.53e-03 & 9.98e-01 & 4.48e-02 & 1.02e+00 & 6.99e-02 & 1.96e+00 & 1.90e+00 & 9.98e-01 & 4.34e-04 & 2.08e+00 \\
        2 & 4.26e-03 & 9.99e-01 & 2.23e-02 & 1.01e+00 & 1.87e-02 & 1.90e+00 & 9.52e-01 & 1.00e+00 & 1.01e-04 & 2.11e+00 \\
        3 & 2.13e-03 & 1.00e+00 & 1.12e-02 & 1.00e+00 & 5.70e-03 & 1.71e+00 & 4.76e-01 & 1.00e+00 & 2.50e-05 & 2.01e+00 \\
        4 & 1.07e-03 & 1.00e+00 & 5.58e-03 & 1.00e+00 & 2.16e-03 & 1.40e+00 & 2.38e-01 & 1.00e+00 & 6.24e-06 & 2.00e+00 \\
        5 & 5.33e-04 & 1.00e+00 & 2.79e-03 & 1.00e+00 & 9.75e-04 & 1.15e+00 & 1.19e-01 & 1.00e+00 & 1.56e-06 & 2.00e+00 \\
        \bottomrule
    \end{tabular}
    \end{center}
    }
\end{table}

\begin{table}[h]
    {\footnotesize
    \caption{Errors and convergence rates using piecewise-constant mortars and a midpoint quadrature rule for error integration for Case 1.}
    \label{tab:convtest-superconv}
    \begin{center}
    \begin{tabular}{l |l l|l l|l l|l l|l l}
        \toprule
        $i$ & $e_{p, D}$ & $r_{p, D}$ & $e_{u, D}$ & $r_{u, D}$ & $e_{p, S}$ & $r_{p, S}$ & $e_{u, S}$ & $r_{u, S}$ & $e_{\lambda}$ & $r_{\lambda}$ \\
        \midrule
        0 & 1.20e-03 &  & 2.50e-02 &  & 2.73e-01 &  & 2.39e-01 &  & 5.12e-03 &  \\
        1 & 2.79e-04 & 2.11e+00 & 6.00e-03 & 2.06e+00 & 6.89e-02 & 1.99e+00 & 6.01e-02 & 1.99e+00 & 1.19e-03 & 2.11e+00 \\
        2 & 7.60e-05 & 1.88e+00 & 1.60e-03 & 1.91e+00 & 1.73e-02 & 1.99e+00 & 1.50e-02 & 2.00e+00 & 3.22e-04 & 1.88e+00 \\
        3 & 1.89e-05 & 2.01e+00 & 4.06e-04 & 1.98e+00 & 4.33e-03 & 2.00e+00 & 3.76e-03 & 2.00e+00 & 8.04e-05 & 2.00e+00 \\
        4 & 4.72e-06 & 2.00e+00 & 1.05e-04 & 1.96e+00 & 1.08e-03 & 2.00e+00 & 9.41e-04 & 2.00e+00 & 2.01e-05 & 2.00e+00 \\
        5 & 1.18e-06 & 2.00e+00 & 2.77e-05 & 1.92e+00 & 2.71e-04 & 2.00e+00 & 2.35e-04 & 2.00e+00 & 5.01e-06 & 2.00e+00 \\
        \bottomrule
    \end{tabular}
    \end{center}
    }
\end{table}

\begin{table}[h]
    {\footnotesize
    \caption{Errors and convergence rates using piecewise-linear mortars and a midpoint quadrature rule for Case 1.}
    \label{tab:convtest-P1mortar-superconv}
    \begin{center}
    \begin{tabular}{l |l l|l l|l l|l l|l l}
        \toprule
        $i$ & $e_{p, D}$ & $r_{p, D}$ & $e_{u, D}$ & $r_{u, D}$ & $e_{p, S}$ & $r_{p, S}$ & $e_{u, S}$ & $r_{u, S}$ & $e_{\lambda}$ & $r_{\lambda}$ \\
        \midrule
        0 & 8.63e-04 &  & 2.04e-02 &  & 2.72e-01 &  & 2.40e-01 &  & 4.20e-03 &  \\
        1 & 2.15e-04 & 2.01e+00 & 5.08e-03 & 2.01e+00 & 6.86e-02 & 1.99e+00 & 6.02e-02 & 1.99e+00 & 1.03e-03 & 2.02e+00 \\
        2 & 5.32e-05 & 2.01e+00 & 1.26e-03 & 2.01e+00 & 1.72e-02 & 2.00e+00 & 1.51e-02 & 2.00e+00 & 2.51e-04 & 2.04e+00 \\
        3 & 1.33e-05 & 2.00e+00 & 3.16e-04 & 2.00e+00 & 4.31e-03 & 2.00e+00 & 3.77e-03 & 2.00e+00 & 6.26e-05 & 2.00e+00 \\
        4 & 3.32e-06 & 2.00e+00 & 7.90e-05 & 2.00e+00 & 1.08e-03 & 2.00e+00 & 9.43e-04 & 2.00e+00 & 1.57e-05 & 2.00e+00 \\
        5 & 8.29e-07 & 2.00e+00 & 1.98e-05 & 2.00e+00 & 2.70e-04 & 2.00e+00 & 2.36e-04 & 2.00e+00 & 3.91e-06 & 2.00e+00 \\
        \bottomrule
    \end{tabular}
    \end{center}
    }
\end{table}

\subsection{Case 2: Porous obstacle}

This test case is inspired by \cite{MAC-MPFA} and
considers a free-flow channel of dimensions $0.75\times0.25$, with a square porous obstacle of dimensions $0.25\times 0.2$ placed halfway on the floor of the channel. It is designed to illustrate the flexibility of the mortar method to use different grids in the two regions in order to resolve local solution features. Flow is enforced from left to right by setting 
$\boldsymbol{\sigma}_S \mathbf{n} \vert_{x = 0} = 1.1 \mathbf{n}$ and 
$\boldsymbol{\sigma}_S \mathbf{n} \vert_{x = 0.75} = \mathbf{n}$ on the left and right boundaries, respectively, while no-flow and no-slip conditions are used on the top and bottom boundaries. We set $\mu = 1$ and $\alpha = 1$. The permeability of the porous medium is set as the following anisotropic tensor:
\begin{equation}
    \mathbf{K} = \mathbf{R}\left(\varphi\right) 
                 \left( 
                    \begin{matrix} 
                        \frac{1}{\beta} k & 0 \\ 
                        0 & k 
                    \end{matrix} 
                 \right)
                 \mathbf{R}^{-1}\left(\varphi\right),
    \quad \mathrm{with} \quad
    \mathbf{R}\left(\varphi\right) = \left(
        \begin{matrix}
            \mathrm{cos}\,\varphi & - \mathrm{sin}\,\varphi \\
            \mathrm{sin}\,\varphi &   \mathrm{cos}\,\varphi            
        \end{matrix}
    \right),
    \label{eq:porous_obstacle_permeability}
\end{equation}
with an anisotropy ratio of $\beta = 100$, $k = 10^{-5}$, and angle $\varphi = \pi/4$.

Figure~\ref{fig:obstacle} shows the velocity and pressure distributions in the domain.
Qualitatively, we see that the flow is partially blocked by the obstacle leading to a high pressure upstream from the block. The anisotropy of the porous medium forces the flow downward and, due the no-flow conditions at the bottom boundary, leads to a high pressure in the lower left triangular region of the obstacle. Along the top of the block, a higher velocity is observed due to the narrowing of the channel. The mesh in Stokes region is graded so that it is finer in the area above the obstacle where the velocity is high. We note that the resulting mismatch between the mesh of the porous medium and the mesh of the free-flow domain introduces no visible artifacts.

\begin{figure}
    \centering
    \begin{subfigure}{0.8\textwidth}
        \includegraphics[width=\textwidth]{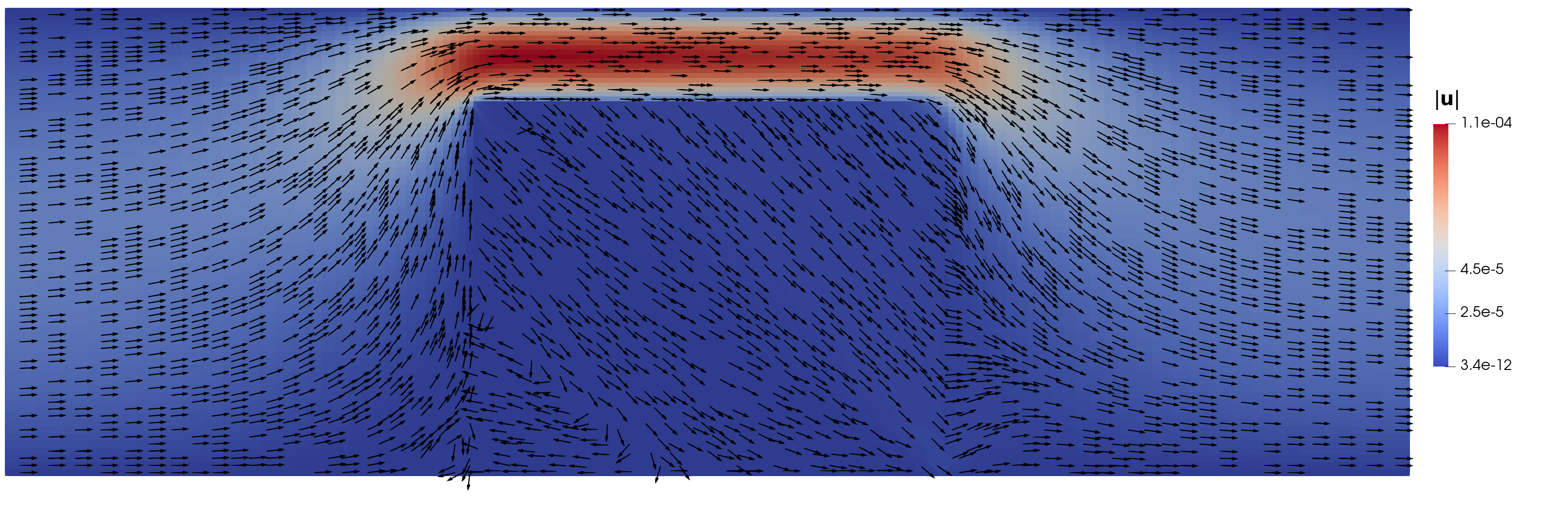}
        \label{fig:obstacle_velocity}
    \end{subfigure}
    \begin{subfigure}{0.8\textwidth}
        \includegraphics[width=\textwidth]{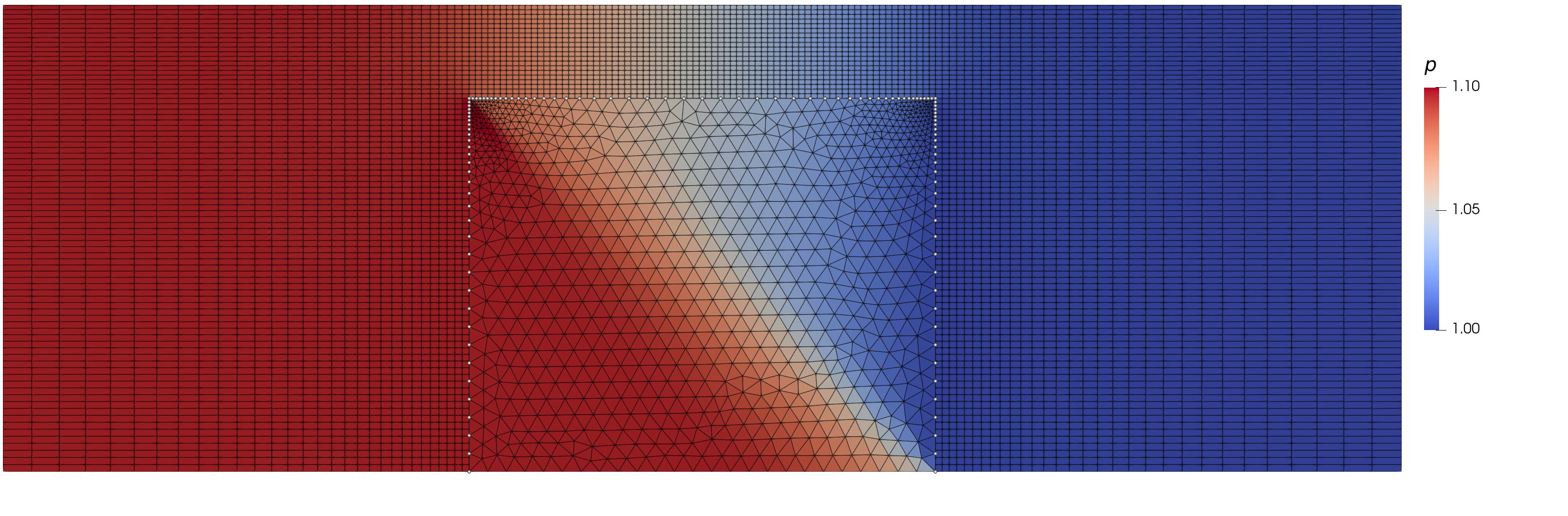}
        \label{fig:obstacle_mesh}
    \end{subfigure}
    \caption{Velocity (top) and pressure together with the mesh (bottom) for Case 2.}
    \label{fig:obstacle}
\end{figure}

\subsection{Case 3: Locally adapted grids}

This test case is motivated by modeling coupled surface and subsurface flows.
The porous medium characterization is inspired by \cite[Example 4]{FluxMortarMPFA} and considers a 
two-dimensional permeability field from the second data set of the
Society of Petroleum Engineers (SPE) Comparative Solution Project SPE10 (see \href{https://www.spe.org/en/csp/}{spe.org/csp/}).
In \cite[Example 4]{FluxMortarMPFA}, the subsurface flow domain is decomposed into $4 \times 4$ subdomains, each of which is discretized with a grid whose refinement reflects the permeability variation in that subdomain. This way, regions with high permeability variations are discretized with finer meshes in comparison with regions where permeability variations are lower. In this example, we take the two center rows of the domain decomposition presented in \cite[Example 4]{FluxMortarMPFA}, flip them vertically, and place a surface flow domain on top. Figure \ref{fig:locally_adapted_mesh} illustrates the permeability field in the porous medium and the meshes in the subdomains. We note that the resulting Stokes and Darcy grids are non-matching along the interface with varying ratio. Moreover, the decomposition of the Darcy domain results in several non-matching Darcy-Darcy interfaces. While the formulation and theory presented in this paper focus on one Stokes and one Darcy subdomain, they can be extended to multiple Stokes and Darcy subdomains using tools developed in \cite{ACWY,APWY,GVY,VasWangYot}.

The entire domain is $6 \times 4.5$ of which the top band with height $1.5$ constitutes the free-flow region. Flow is enforced from left to right along the fluid region by imposing
$\boldsymbol{\sigma}_S \mathbf{n} \vert_{x = 0} = \mathbf{n}$ and 
$\boldsymbol{\sigma}_S \mathbf{n} \vert_{x = 6} = \mathbf{0}$ on the left and right boundaries, respectively. At the bottom of the porous medium, a fixed pressure of $p = 0$ is 
used to also drive the flow downwards through the porous medium. On all remaining boundaries,
no-slip and/or no-flow boundary conditions are applied. We set $\mu = 1$ and $\alpha = 1$.

A visualization of the velocity distribution in the domain is shown in Figure~\ref{fig:locally_adapted_velocity}.
We once again observe a qualitatively good fit with the expected behavior of the system. The majority of the flow infiltrates the porous medium in the first half of the domain and then follows the high-permeable regions to the bottom boundary. The locally refined grids accurately capture the channelized flow field while the coarser grids in the low-permeable regions allow for a reduction in computational cost.

\begin{figure}
    \centering
    \includegraphics[width=0.8\textwidth]{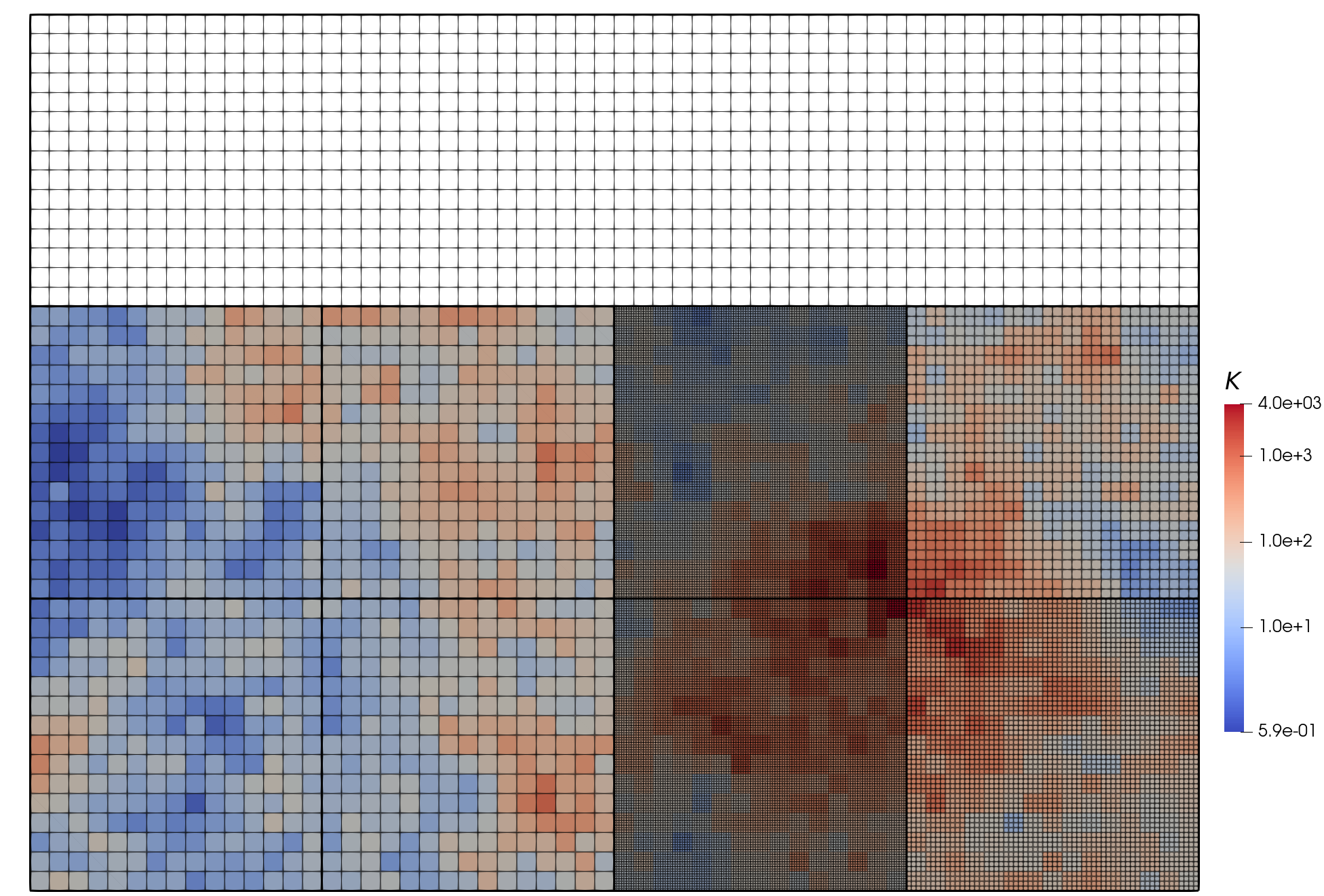}
    \caption{Visualization of the mesh over the entire domain and the permeability distribution used in the porous medium for Case 3.}
    \label{fig:locally_adapted_mesh}
\end{figure}

\begin{figure}
    \centering
    \includegraphics[width=0.8\textwidth]{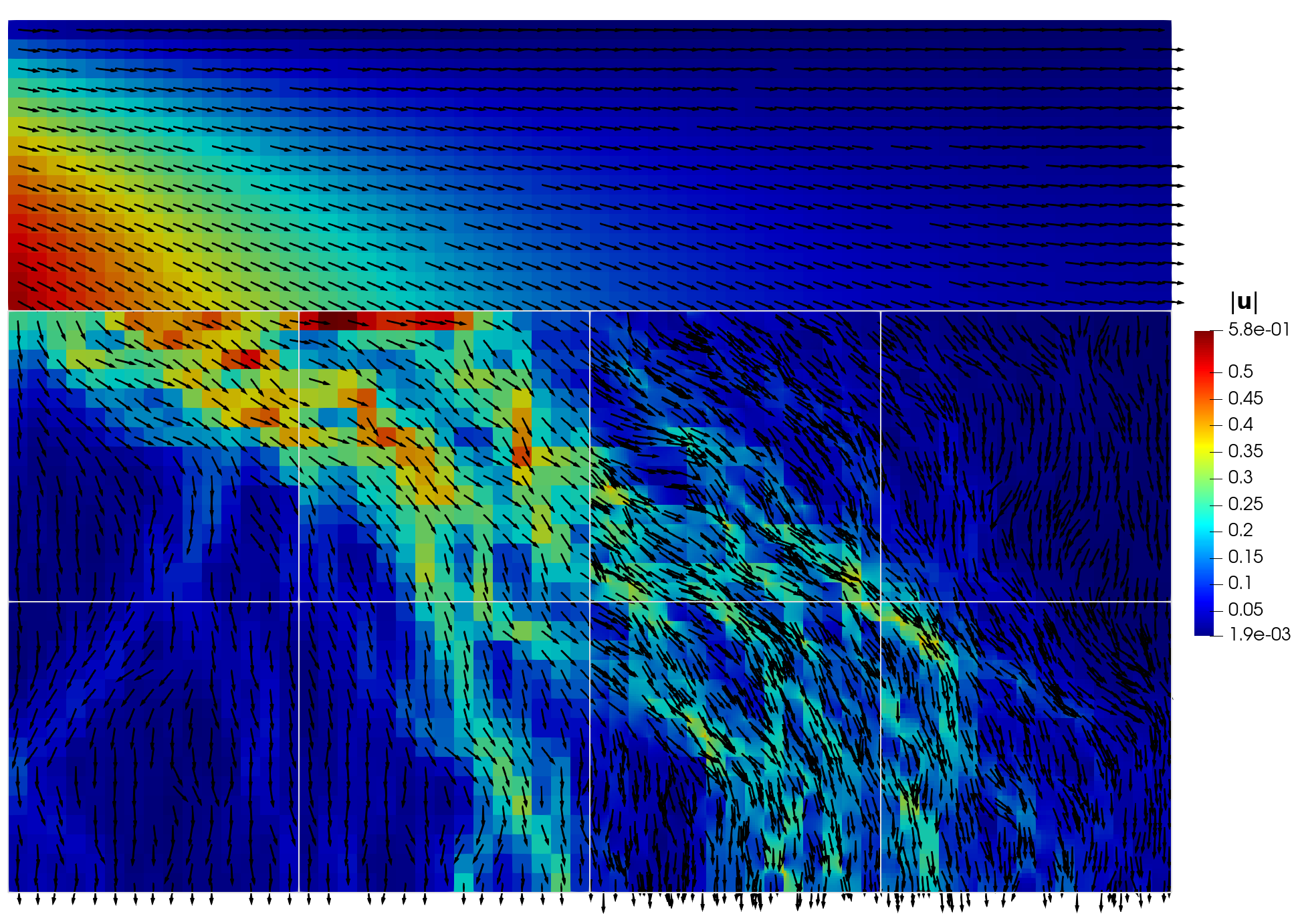}
    \caption{Velocity distribution for Case 3.}
    \label{fig:locally_adapted_velocity}
\end{figure}

\section{Conclusions}\label{sec:conclusions}

We presented a numerical method for coupled Stokes-Darcy flows that exhibits local mass and momentum conservation and allows for non-matching grids on the interface. The method combines the MAC scheme for Stokes, the RT$_0$ MFE method for Darcy, and mortar finite elements on the interface. We established well posedness and first order convergence of the method. We further presented a non-overlapping domain decomposition algorithm for the solution of the resulting coupled algebraic problem, which requires solving only decoupled subdomain problems and can result in scalable parallel implementations. A numerical test was presented to verify the theoretical convergence results. Second order superconvergence was also observed in suitable discrete norms. Finally, two computational experiments for challenging problems were presented to illustrate the applicability and flexibility of the proposed method.

\section*{Data availability}
Data are available at the following repositories:\\
git repository: https://git.iws.uni-stuttgart.de/dumux-pub/boon2023a;\\
source code: https://doi.org/10.18419/darus-3598;\\
results: https://doi.org/10.18419/darus-3599.

\section*{Acknowledgments}
This project has received funding from the European Union's Horizon 2020 research and innovation programme under the Marie Sk\l odowska-Curie grant agreement No. 101031434 -- MiDiROM, from the Deutsche Forschungsgemeinschaft (DFG, German Research Foundation) under SFB 1313, Project Number 327154368, from University of Stuttgart Cluster of Excellence SimTech, and from the U.S. National Science Foundation under grant DMS 2111129.

\section*{Declarations}
The authors have no competing interests to declare that are relevant to the content of this article.

\bibliographystyle{abbrv}
\bibliography{SD-MAC-MFE}

\begin{thebibliography}{10}

\bibitem{MPFA}
I.~Aavatsmark, T.~Barkve, O.~B{o}e, and T.~Mannseth.
\newblock Discretization on unstructured grids for inhomogeneous, anisotropic
  media. {I}. {D}erivation of the methods.
\newblock {\em SIAM J. Sci. Comput.}, 19(5):1700--1716, 1998.

\bibitem{AEKWY}
I.~Aavatsmark, G.~T. Eigestad, R.~A. Klausen, M.~F. Wheeler, and I.~Yotov.
\newblock Convergence of a symmetric {MPFA} method on quadrilateral grids.
\newblock {\em Comput. Geosci.}, 11(4):333--345, 2007.

\bibitem{ambartsumyan2019nonlinear}
I.~Ambartsumyan, V.~J. Ervin, T.~Nguyen, and I.~Yotov.
\newblock A nonlinear {S}tokes-{B}iot model for the interaction of a
  non-{N}ewtonian fluid with poroelastic media.
\newblock {\em ESAIM Math. Model. Numer. Anal.}, 53(6):1915--1955, 2019.

\bibitem{ArbBrun}
T.~Arbogast and D.~S. Brunson.
\newblock A computational method for approximating a {D}arcy-{S}tokes system
  governing a vuggy porous medium.
\newblock {\em Comput. Geosci.}, 11(3):207--218, 2007.

\bibitem{ACWY}
T.~Arbogast, L.~C. Cowsar, M.~F. Wheeler, and I.~Yotov.
\newblock Mixed finite element methods on nonmatching multiblock grids.
\newblock {\em SIAM J. Numer. Anal.}, 37(4):1295--1315, 2000.

\bibitem{APWY}
T.~Arbogast, G.~Pencheva, M.~F. Wheeler, and I.~Yotov.
\newblock A multiscale mortar mixed finite element method.
\newblock {\em Multiscale Model. Simul.}, 6(1):319--346, 2007.

\bibitem{Bernardi-etal-mortar-SD}
C.~Bernardi, T.~C. Rebollo, F.~Hecht, and Z.~Mghazli.
\newblock Mortar finite element discretization of a model coupling {D}arcy and
  {S}tokes equations.
\newblock {\em M2AN Math. Model. Numer. Anal.}, 42(3):375--410, 2008.

\bibitem{boon2020parameter}
W.~M. Boon.
\newblock A parameter-robust iterative method for {S}tokes-{D}arcy problems
  retaining local mass conservation.
\newblock {\em ESAIM Math. Model. Numer. Anal.}, 54(6):2045--2067, 2020.

\bibitem{flux-mortar}
W.~M. Boon, D.~Gl\"{a}ser, R.~Helmig, and I.~Yotov.
\newblock Flux-mortar mixed finite element methods on nonmatching grids.
\newblock {\em SIAM J. Numer. Anal.}, 60(3):1193--1225, 2022.

\bibitem{FluxMortarMPFA}
W.~M. Boon, D.~Gl\"{a}ser, R.~Helmig, and I.~Yotov.
\newblock Flux-mortar mixed finite element methods with multipoint flux
  approximation.
\newblock {\em Comput. Methods Appl. Mech. Engrg.}, 405:Paper No. 115870, 28,
  2023.

\bibitem{Brezzi-Fortin}
F.~Brezzi and M.~Fortin.
\newblock {\em {Mixed and hybrid finite element methods}}.
\newblock Springer-Verlag, New York, 1991.

\bibitem{Chen-Gunz-Robin}
W.~Chen, M.~Gunzburger, F.~Hua, and X.~Wang.
\newblock A parallel {R}obin-{R}obin domain decomposition method for the
  {S}tokes-{D}arcy system.
\newblock {\em SIAM J. Numer. Anal.}, 49(3):1064--1084, 2011.

\bibitem{DMQ}
M.~Discacciati, E.~Miglio, and A.~Quarteroni.
\newblock Mathematical and numerical models for coupling surface and
  groundwater flows.
\newblock {\em Appl. Numer. Math.}, 43(1-2):57--74, 2002.
\newblock 19th Dundee Biennial Conference on Numerical Analysis (2001).

\bibitem{Disc-Quart-2003}
M.~Discacciati and A.~Quarteroni.
\newblock Analysis of a domain decomposition method for the coupling of
  {S}tokes and {D}arcy equations.
\newblock In {\em Numerical mathematics and advanced applications}, pages
  3--20. Springer Italia, Milan, 2003.

\bibitem{Disc-Quart-2004}
M.~Discacciati and A.~Quarteroni.
\newblock Convergence analysis of a subdomain iterative method for the finite
  element approximation of the coupling of {S}tokes and {D}arcy equations.
\newblock {\em Comput. Vis. Sci.}, 6(2-3):93--103, 2004.

\bibitem{Disc-Quart-Valli-2007}
M.~Discacciati, A.~Quarteroni, and A.~Valli.
\newblock Robin-{R}obin domain decomposition methods for the {S}tokes-{D}arcy
  coupling.
\newblock {\em SIAM J. Numer. Anal.}, 45(3):1246--1268 (electronic), 2007.

\bibitem{Edwards-Rogers}
M.~G. Edwards and C.~F. Rogers.
\newblock Finite volume discretization with imposed flux continuity for the
  general tensor pressure equation.
\newblock {\em Comput. Geosci.}, 2(4):259--290 (1999), 1998.

\bibitem{Eymard-MAC-rectangles}
R.~Eymard, T.~Gallou\"{e}t, R.~Herbin, and J.-C. Latch\'{e}.
\newblock Convergence of the {MAC} scheme for the compressible {S}tokes
  equations.
\newblock {\em SIAM J. Numer. Anal.}, 48(6):2218--2246, 2010.

\bibitem{Galvis-Sarkis}
J.~Galvis and M.~Sarkis.
\newblock Non-matching mortar discretization analysis for the coupling
  {S}tokes-{D}arcy equations.
\newblock {\em Electron. Trans. Numer. Anal.}, 26:350--384, 2007.

\bibitem{Galvis-Sarkis-DD}
J.~Galvis and M.~Sarkis.
\newblock F{ETI} and {BDD} preconditioners for {S}tokes-{M}ortar-{D}arcy
  systems.
\newblock {\em Commun. Appl. Math. Comput. Sci.}, 5:1--30, 2010.

\bibitem{Gatica-09}
G.~N. Gatica, S.~Meddahi, and R.~Oyarz{\'u}a.
\newblock A conforming mixed finite-element method for the coupling of fluid
  flow with porous media flow.
\newblock {\em IMA J. Numer. Anal.}, 29(1):86--108, 2009.

\bibitem{Gatica-11}
G.~N. Gatica, R.~Oyarz{\'u}a, and F.-J. Sayas.
\newblock Analysis of fully-mixed finite element methods for the
  {S}tokes-{D}arcy coupled problem.
\newblock {\em Math. Comp.}, 80(276):1911--1948, 2011.

\bibitem{Gir-Lop}
V.~Girault and H.~Lopez.
\newblock Finite-element error estimates for the {MAC} scheme.
\newblock {\em IMA J. Numer. Anal.}, 16(3):247--379, 1996.

\bibitem{Girault-Raviart}
V.~Girault and P.-A. Raviart.
\newblock {\em Finite element methods for {N}avier-{S}tokes equations},
  volume~5 of {\em Springer Series in Computational Mathematics}.
\newblock Springer-Verlag, Berlin, 1986.
\newblock Theory and algorithms.

\bibitem{GVY}
V.~Girault, D.~Vassilev, and I.~Yotov.
\newblock Mortar multiscale finite element methods for {S}tokes-{D}arcy flows.
\newblock {\em Numer. Math.}, 127(1):93--165, 2014.

\bibitem{Han-Wu}
H.~Han and X.~Wu.
\newblock A new mixed finite element formulation and the {MAC} method for the
  {S}tokes equations.
\newblock {\em SIAM J. Numer. Anal.}, 35(2):560--571, 1998.

\bibitem{Harlow-Welch}
F.~H. Harlow and J.~E. Welch.
\newblock Numerical calculation of time-dependent viscous incompressible flow
  of fluid with free surface.
\newblock {\em Phys. Fluids}, 8(12):2182--2189, 1965.

\bibitem{Ing-Whe-Yot}
R.~Ingram, M.~Wheeler, and I.~Yotov.
\newblock A multipoint flux mixed finite element method on hexahedra.
\newblock {\em SIAM J. Numer. Anal.}, 48(4):1281--1312, 2010.

\bibitem{Kanschat}
G.~Kanschat.
\newblock Divergence-free discontinuous {G}alerkin schemes for the {S}tokes
  equations and the {MAC} scheme.
\newblock {\em Internat. J. Numer. Methods Fluids}, 56(7):941--950, 2008.

\bibitem{KanRiv}
G.~Kanschat and B.~Rivi{\`e}re.
\newblock A strongly conservative finite element method for the coupling of
  {S}tokes and {D}arcy flow.
\newblock {\em J. Comput. Phys.}, 229(17):5933--5943, 2010.

\bibitem{Kar-Mar-Win}
T.~Karper, K.-A. Mardal, and R.~Winther.
\newblock Unified finite element discretizations of coupled {D}arcy-{S}tokes
  flow.
\newblock {\em Numer. Methods Partial Differential Equations}, 25(2):311--326,
  2009.

\bibitem{LSY}
W.~J. Layton, F.~Schieweck, and I.~Yotov.
\newblock Coupling fluid flow with porous media flow.
\newblock {\em SIAM J. Numer. Anal.}, 40(6):2195--2218 (2003), 2002.

\bibitem{Li-Sun}
J.~Li and S.~Sun.
\newblock The superconvergence phenomenon and proof of the {MAC} scheme for the
  {S}tokes equations on non-uniform rectangular meshes.
\newblock {\em J. Sci. Comput.}, 65(1):341--362, 2015.

\bibitem{Li-Rui}
X.~Li and H.~Rui.
\newblock Superconvergence of {MAC} scheme for a coupled free flow-porous media
  system with heat transport on non-uniform grids.
\newblock {\em J. Sci. Comput.}, 90(3):Paper No. 90, 32, 2022.

\bibitem{Nic}
R.~A. Nicolaides.
\newblock Analysis and convergence of the {MAC} scheme. {I}. {T}he linear
  problem.
\newblock {\em SIAM J. Numer. Anal.}, 29(6):1579--1591, 1992.

\bibitem{Nic-Wu}
R.~A. Nicolaides and X.~Wu.
\newblock Analysis and convergence of the {MAC} scheme. {II}. {N}avier-{S}tokes
  equations.
\newblock {\em Math. Comp.}, 65(213):29--44, 1996.

\bibitem{RT}
R.~Raviart and J.~Thomas.
\newblock A mixed finite element method for 2nd order elliptic problems.
\newblock In {\em Mathematical Aspects of the Finite Element Method, Lecture
  Notes in Mathematics}, volume 606, pages 292--315. Springer-Verlag, New York,
  1977.

\bibitem{RivYot}
B.~Rivi\`{e}re and I.~Yotov.
\newblock Locally conservative coupling of {S}tokes and {D}arcy flows.
\newblock {\em SIAM J. Numer. Anal.}, 42(5):1959--1977, 2005.

\bibitem{Rui-Sun}
H.~Rui and Y.~Sun.
\newblock A {MAC} scheme for coupled {S}tokes-{D}arcy equations on non-uniform
  grids.
\newblock {\em J. Sci. Comput.}, 82(3):Paper No. 79, 29, 2020.

\bibitem{MAC-MPFA}
M.~Schneider, K.~Weishaupt, D.~Gl\"{a}ser, W.~M. Boon, and R.~Helmig.
\newblock Coupling staggered-grid and {MPFA} finite volume methods for free
  flow/porous-medium flow problems.
\newblock {\em J. Comput. Phys.}, 401:109012, 17, 2020.

\bibitem{MAC-SD}
M.-C. Shiue, K.~C. Ong, and M.-C. Lai.
\newblock Convergence of the {MAC} scheme for the {S}tokes/{D}arcy coupling
  problem.
\newblock {\em J. Sci. Comput.}, 76(2):1216--1251, 2018.

\bibitem{Song-Wang-Yotov}
P.~Song, C.~Wang, and I.~Yotov.
\newblock Domain decomposition for {S}tokes-{D}arcy flows with curved
  interfaces.
\newblock {\em Procedia Computer Science}, 18:1077--1086, 2013.

\bibitem{Song-Yotov-vegas}
P.~Song and I.~Yotov.
\newblock Coupling surface and subsurface flows with curved interfaces.
\newblock {\em Contemporary Mathematics}, 586:331--339, 2013.

\bibitem{VasWangYot}
D.~Vassilev, C.~Wang, and I.~Yotov.
\newblock Domain decomposition for coupled {S}tokes and {D}arcy flows.
\newblock {\em Comput. Methods Appl. Mech. Engrg.}, 268:264--283, 2014.

\bibitem{vassilev-2009}
D.~Vassilev and I.~Yotov.
\newblock Coupling stokes–darcy flow with transport.
\newblock {\em SIAM Journal on Scientific Computing}, 31(5):3661--3684, 2009.

\bibitem{WheXueYot-msmortar}
M.~F. Wheeler, G.~Xue, and I.~Yotov.
\newblock A multiscale mortar multipoint flux mixed finite element method.
\newblock {\em ESAIM Math. Model. Numer. Anal.}, 46(4):759--796, 2012.

\bibitem{WY-MPFA}
M.~F. Wheeler and I.~Yotov.
\newblock A multipoint flux mixed finite element method.
\newblock {\em SIAM J. Numer. Anal.}, 44(5):2082--2106, 2006.

\end{thebibliography}

\end{document}